\documentclass{bau}
\usepackage{amssymb}
\usepackage{amsmath}
\usepackage{enumerate}
\usepackage{tikz}
\usepackage[colorlinks=true,linkcolor=magenta,citecolor=magenta]{hyperref}
\theoremstyle{plain}
\newtheorem{theorem}{Theorem}[section]
\newtheorem{lemma}[theorem]{Lemma}
\newtheorem{proposition}[theorem]{Proposition}
\newtheorem{corollary}[theorem]{Corollary}
\newtheorem{alemma}{Lemma}

\theoremstyle{definition}
\newtheorem{example}[theorem]{Example}
\theoremstyle{remark}
\newtheorem*{claim}{Claim}

\begin{document}

\title[Lattices with complementation]{Varieties and quasivarieties of lattices\\ with complementation}
\author[V. Cenker]{V\'aclav Cenker} 
\address{Palack\'y University Olomouc, 
Faculty of Science,
Department of Algebra and Geometry,
17.\ listopadu 12,
771 46 Olomouc,
Czechia}
\email{vaclav.cenker01@upol.cz}
\author[I. Chajda]{Ivan Chajda} 
\address{Palack\'y University Olomouc, 
Faculty of Science,
Department of Algebra and Geometry,
17.\ listopadu 12,
771 46 Olomouc,
Czechia}
\email{ivan.chajda@upol.cz}
\author[J. K\"uhr]{Jan K\"uhr}
\address{Palack\'y University Olomouc, 
Faculty of Science,
Department of Algebra and Geometry,
17.\ listopadu 12,
771 46 Olomouc,
Czechia}
\email{jan.kuhr@upol.cz}
\author[H. L\"anger]{Helmut L\"anger}
\address{TU Wien,
Faculty of Mathematics and Geoinformation,
Institute of Discrete Mathematics and Geometry,
Wiedner Hauptstra\ss e 8--10,
1040 Vienna,
Austria, 
and
Palack\'y University Olomouc, 
Faculty of Science,
Department of Algebra and Geometry,
17.\ listopadu 12,
771 46 Olomouc,
Czechia}
\email{helmut.laenger@tuwien.ac.at}

\begin{abstract}
We investigate (quasi)varieties of lattices with complementation, i.e., complemented lattices equipped with a fixed complementation as a unary operation. 
We focus on subclasses satisfying additional conditions, such as the quasi-identity
$({x'\wedge y\approx 0} \;\&\; {x\wedge y'\approx 0})$ $\Rightarrow x\approx y$,
modularity, or De Morgan's laws.
We present a construction resembling a semidirect product that yields infinitely many finite subdirectly irreducible modular lattices with complementation satisfying this quasi-identity. We axiomatize small varieties, each of which covers the variety of Boolean algebras, generated by certain small modular lattices with De~Morgan complementation.
\end{abstract}



\maketitle

\section{Introduction}
\label{SEC1}

By a \emph{lattice with complementation} we mean an algebra $\mathbf{L}=(L,\vee,\wedge,{}',0,1)$ of type $(2,2,1,0,0)$ such that $(L,\vee,\wedge,0,1)$ is a bounded lattice and, for every $a\in L$, $a'$ is a complement of $a$. 
Loosely speaking, this is simply a complemented lattice equipped with a designated complementation as a unary operation.

In this paper, we focus on lattices with complementation satisfying certain additional conditions that significantly affect the interaction between the lattice structure and the chosen complementation. This can already be illustrated by the smallest nondistributive complemented lattices (see Figure~\ref{F1}).
The diamond admits two complementations (up to isomorphism), namely $a\mapsto b\mapsto c\mapsto a$ and $a\mapsto b\mapsto c\mapsto b$.
In the former case, the quasi-identity
\[
\tag{W}\label{W}
(x'\wedge y\approx 0 \;\mathrel{\&}\; x\wedge y'\approx 0)
\;\Rightarrow\; x\approx y
\]
as well as De Morgan's laws
\[
\tag{DM}\label{DM}
(x\vee y)'\approx x'\wedge y'
\quad\text{and}\quad
(x\wedge y)'\approx x'\vee y'
\]
are satisfied, whereas in the latter case neither \eqref{W} nor \eqref{DM} holds
(indeed, $a'\wedge c=b\wedge c=0$ and $a\wedge c'=a\wedge b=0$ witness the failure of \eqref{W}, while $(a\vee c)'=0$ and $a'\wedge c'=b$ witness the failure of De Morgan's laws).
The pentagon also admits two complementations, namely
$a\mapsto c\mapsto b\mapsto c$ and
$b\mapsto c\mapsto a\mapsto c$.
In both cases, \eqref{W} fails (as $a'\wedge b=c\wedge b=0$ and $a\wedge b'=a\wedge c=0$), but \eqref{DM} is satisfied.

\begin{figure}[ht]
\centering
\begin{tikzpicture}[scale=0.8,radius=2pt,font=\small]
\filldraw (1,0) circle node [below] {$0$} coordinate (0);
\filldraw (0,1) circle node [left] {$a$} coordinate (a);
\filldraw (1,1) circle node [right] {$b$} coordinate (b);
\filldraw (2,1) circle node [right] {$c$} coordinate (c);
\filldraw (1,2) circle node [above] {$1$} coordinate (1);
\draw (0) -- (a) -- (1) -- (c) -- (0) -- (1);
\end{tikzpicture}
\qquad
\begin{tikzpicture}[scale=0.8,radius=2pt] 
\filldraw (1,0) circle node [below] {$0$} coordinate (0);
\filldraw (0.2,0.6) circle node [left] {$a$} coordinate (a);
\filldraw (0.2,1.4) circle node [left] {$b$} coordinate (b);
\filldraw (1.8,1) circle node [right] {$c$} coordinate (c);
\filldraw (1,2) circle node [above] {$1$} coordinate (1);
\draw (0) -- (a) -- (b) -- (1) -- (c) -- (0);
\end{tikzpicture}
\caption{The diamond and the pentagon}\label{F1}
\end{figure}
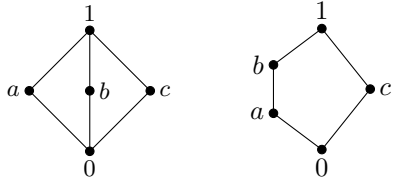

Thus, we study lattices with complementation that satisfy the quasi-identity \eqref{W}, are modular, or satisfy De Morgan's laws \eqref{DM}.
Accordingly, we use $\mathcal{W}$ and $\mathcal{M}$ to denote, respectively, the class of lattices with complementation satisfying \eqref{W} and the class of modular lattices with complementation.
Moreover, given a class $\mathcal{K}$ of lattices with complementation, we let $\mathcal{K}_{DM}$ denote the subclass consisting of those members of $\mathcal{K}$ which satisfy De Morgan's laws \eqref{DM}.
In the case of lattices with complementation satisfying \eqref{DM}, we often speak of \emph{lattices with De Morgan complementation} (note, however, that such a complementation need not be involutive).

The paper is organized as follows.
In Section~\ref{SEC2}, we collect several basic examples and establish first properties of lattices with complementation that satisfy some of the above-mentioned conditions. The most prominent examples are the lattices with complementation $\mathbf{M}_n'$ (Example~\ref{EXA1}).
In Section~\ref{SEC3}, we present a construction that roughly resembles a semidirect product of a Boolean algebra with a complemented lattice (Proposition~\ref{L:V1}), and as a consequence we show that there are infinitely many finite subdirectly irreducible algebras within the class $\mathcal{W}$ (Corollary~\ref{C:34}). 
In Section~\ref{SEC4}, we investigate members of the classes $\mathcal{W}_{DM}$ and $\mathcal{M}_{DM}$ in connection with the so-called (dual) weak orthomodularity, and we characterize their neutral elements in terms of complementation.
In Section~\ref{SEC5}, we describe some congruence properties of the varieties $\mathcal{M}$ and $\mathcal{M}_{DM}$ that will be needed in Section~\ref{SEC6}, where we axiomatize the varieties generated by the lattices with complementation $\mathbf{M}_n'$ (Theorems~\ref{T:M3axio} and \ref{T:Mnaxio}) and describe the free algebras in these varieties (Theorem~\ref{T:FREE}). We also show that $\mathcal{M}_{DM}$ has uncountably many subvarieties (Theorem~\ref{T:unc}).

\section{Examples and first properties}
\label{SEC2}

As usual, for any cardinal $\kappa\ge 2$, we use $\mathbf{M}_\kappa$ to denote the unique bounded lattice of length $2$ and width $\kappa$; thus the set
$M_\kappa\setminus\{0,1\}=\{a_\lambda\mid \lambda<\kappa\}$
forms an antichain of size (width) $\kappa$.
We do not exclude the case $\kappa=2$, although the Boolean lattice $\mathbf{M}_2$ is of limited interest in the present work.

We have already mentioned that our guiding example is the lattice $\mathbf{M}_3$ equipped with the complementation that is defined by cyclically permuting the elements $a_0,a_1,a_2$ (denoted $a,b,c$ in Figure~\ref{F1} of Section~\ref{SEC1}). 
We denote this lattice with complementation by $\mathbf{M}_3'$.
More formally, and more generally:

\begin{figure}[ht]
\centering
\begin{tikzpicture}[scale=0.8,radius=2pt,font=\small]
\filldraw (2,0) circle node [below] {$0$} coordinate (0);
\filldraw (2,2) circle node [above] {$1$} coordinate (1);
\filldraw (0,1) circle node [left] {$a_0$} coordinate (A);
\filldraw (1,1) circle node [left] {$a_1$} coordinate (B);
\filldraw (4,1) circle node [right] {$a_{n-1}$} coordinate (C);
\node at (2.75,1) {$\cdots$};
\draw (0) -- (A) -- (1) -- (B) -- (0) -- (C) -- (1);
\draw (1) -- (2.5,1.5);
\draw (0) -- (2.5,0.5);
\end{tikzpicture}
\caption{The lattice $\mathbf{M}_n$ for $n\ge 2$}\label{F2}
\end{figure}
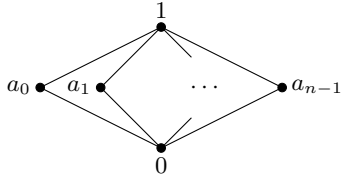

\begin{example}\label{EXA1}
For any integer $n\ge 2$, we use $\mathbf{M}_n'$ to denote the lattice $\mathbf{M}_n$ (see Figure~\ref{F2}) equipped with the complementation defined by the cycle $a_0\mapsto a_1\mapsto\dots\mapsto a_{n-1}\mapsto a_0$ on $M_n\setminus\{0,1\}=\{a_0,\dots,a_{n-1}\}$, which is usually written as $(a_0\;a_1\;\dots\;a_{n-1})$.
Thus, $\mathbf{M}_n'=(M_n,\vee,\wedge,{}',0,1)$, where 
\[
a_i':=a_{i+1}
\]
for every $i\in\{0,\dots,n-1\}$, with $i+1$ taken modulo $n$.
Note that $\mathbf{M}_n'$ satisfies De~Morgan's laws \eqref{DM}, because $a_i'=a_{i+1}\ne a_{j+1}=a_j'$ whenever $a_i\ne a_j$.

The concrete choice of the cycle is not essential, because all cycles of length $n$ yield the same (up to isomorphism) lattice with complementation (cf.\ Lemma~\ref{lem1}).
So, for example, we could have alternatively defined $\mathbf{M}_3'$ by means of the cycle $(a_2\;a_1\;a_0)$.

The Boolean algebra $\mathbf{M}_2'\cong\mathbf{2}^2$ certainly belongs to the quasivariety $\mathcal{W}$, and so does $\mathbf{M}_3'$, as can be easily verified. 
However, no $\mathbf{M}_n'$ with $n\ge 4$ belongs to $\mathcal{W}$, because $a_0'\wedge a_2=a_1\wedge a_2=0$ and $a_0\wedge a_2'=a_0\wedge a_3=0$, yet $a_0\ne a_2$.
We will see in Proposition \ref{L:H2} that no $\mathbf{M}_\kappa$ with (finite or infinite) $\kappa\ge 4$ admits a complementation making it a member of $\mathcal{W}$.

For $n\ge 4$, there are other (injective) De Morgan complementations on $\mathbf{M}_n$. Each such complementation is determined by a product of several pairwise disjoint cycles of length $\ge 2$ on $\{a_0,\dots,a_{n-1}\}$, say $\sigma_1\circ\dots\circ\sigma_r$. Then $\mathbf{M}_n$ with this complementation is isomorphic (as a lattice with complementation) to the horizontal sum of $\mathbf{M}_{p_1}',\dots,\mathbf{M}_{p_r}'$, where $p_i$ is the length of the respective cycle $\sigma_i$. Therefore, it is natural to focus on complementations defined by single cycles.
\end{example}

It is worth noticing that, for any pair of integers $n,p\ge 2$, the lattice $\mathbf{M}_p$ satisfies J\'onsson's identity
\[
\tag{$\mathrm{M}_n$}\label{MN1}
x_0\wedge\bigwedge_{1\le i<j\le n} (x_i\vee x_j) \le 
\bigvee_{1\le i\le n} (x_0\wedge x_i)
\]
if and only if $p\le n$ (see \cite{Jo}). In particular, only the lattices $\mathbf{M}_2$ and $\mathbf{M}_3$ satisfy the identity
\[
\tag{$\mathrm{M}_3$}\label{M31}
u\wedge (x\vee y)\wedge (x\vee z)\wedge (y\vee z)
\le (u\wedge x)\vee (u\wedge y)\vee (u\wedge z).
\]
In Section~\ref{SEC6}, these identities will be employed in the axiomatization of the varieties generated by the lattices with complementation $\mathbf{M}_n'$ (for $n\ge 3$).

We should point out that the quasivariety $\mathcal{W}$ is not closed under horizontal sums, i.e., the quasi-identity \eqref{W} is not preserved by forming horizontal sums, because when $a\neq b$ (both distinct from $0$ and $1$) are elements from two distinct constituent lattices, then $a'\wedge b=0$ and $a\wedge b'=0$.

Now, we describe all members of $\mathcal{W}$ of size $\le 16$.
In addition to the small Boolean algebras $\mathbf{2}^k$ with $k\le 4$, we have already encountered $\mathbf{M}_3'$, which is the only member of $\mathcal{W}$ of size $5$ (see Example~\ref{EXA1} and the introduction).
Then there are two members of size $10$ and other two of size $16$, as the following two examples show.

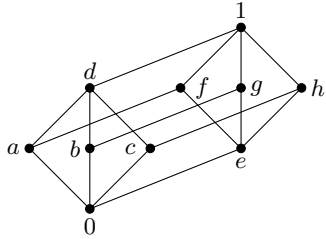
\begin{figure}[ht]
\centering
\begin{tikzpicture}[scale=0.8,radius=2pt,font=\small]
\filldraw (1,0) circle node [below] {$0$} coordinate (0);
\filldraw (0,1) circle node [left] {$a$} coordinate (a);
\filldraw (1,1) circle node [left] {$b$} coordinate (b);
\filldraw (2,1) circle node [left=2pt] {$c$} coordinate (c);
\filldraw (1,2) circle node [above] {$d$} coordinate (d);
\filldraw (3.5,1) circle node [below] {$e$} coordinate (e);
\filldraw (2.5,2) circle node [right=2pt] {$f$} coordinate (f);
\filldraw (3.5,2) circle node [right] {$g$} coordinate (g);
\filldraw (4.5,2) circle node [right] {$h$} coordinate (h);
\filldraw (3.5,3) circle node [above] {$1$} coordinate (1);
\draw (0) -- (b) -- (d) -- (a) -- (0) -- (c) -- (d);
\draw (e) -- (g) -- (1) -- (f) -- (e) -- (h) -- (1);
\draw (0) -- (e);
\draw (a) -- (f); 
\draw (b) -- (g);
\draw (c) -- (h);
\draw (d) -- (1);
\end{tikzpicture}
\caption{The lattice from Example \ref{EXA2}}\label{F3}
\end{figure}

\begin{example}[cf.\ \cite{CL18}, Examples 2.3 and 3.4]
\label{EXA2}
Consider the lattice in Figure~\ref{F3}.
Up to isomorphism, there are two complementations that turn this lattice into a lattice with complementation belonging to $\mathcal{W}$. We will elaborate on this in Example \ref{EXA4}, and more generally in Section \ref{SEC3}, where we will describe complementations on direct products of complemented lattices with Boolean lattices.

The first possibility is given as follows:
\begin{center}
\begin{tabular}{c|cccccccc}
$x$ & $a$ & $b$ & $c$ & $d$ & $e$ & $f$ & $g$ & $h$ \\
\hline
$x'$ & $g$ & $h$ & $f$ & $e$ & $d$ & $b$ & $c$ & $a$
\end{tabular}
\end{center}
It is easy to see that the lattice with complementation so defined, $\mathbf{H}_1$ say, is isomorphic to the direct product $\mathbf{M}_3'\times\mathbf{2}$ under 
$0\mapsto (0,0)$, $a\mapsto (a_0,0)$, $b\mapsto (a_1,0)$, $c\mapsto (a_2,0)$, $d\mapsto (1,0)$, $e\mapsto (0,1)$, $f\mapsto (a_0,1)$, $g\mapsto (a_1,1)$, $h\mapsto (a_2,1)$, and $1\mapsto (1,1)$.
Hence $\mathbf{H}_1$ satisfies De Morgan's laws \eqref{DM}.

The other possibility is involutive:
\begin{center}
\begin{tabular}{c|cccccccc}
$x$ & $a$ & $b$ & $c$ & $d$ & $e$ & $f$ & $g$ & $h$ \\
\hline
$x'$ & $g$ & $h$ & $f$ & $e$ & $d$ & $c$ & $a$ & $b$
\end{tabular}
\end{center}
In contrast to $\mathbf{H}_1$, the lattice with complementation so defined, $\mathbf{H}_2$ say, does not satisfy De Morgan's laws (e.g., $a\le f$ but $a'=g\ngeq c=f'$).
However, it should be noted that $\mathbf{H}_2$ is subdirectly irreducible, since its only nontrivial congruence is 
\[
\theta_0=\{0,a,b,c,d\}^2\cup\{e,f,g,h,1\}^2.
\]
Up to isomorphism, $\mathbf{H}_1$ and $\mathbf{H}_2$ are the only members of $\mathcal{W}$ of size $10$.
\end{example}

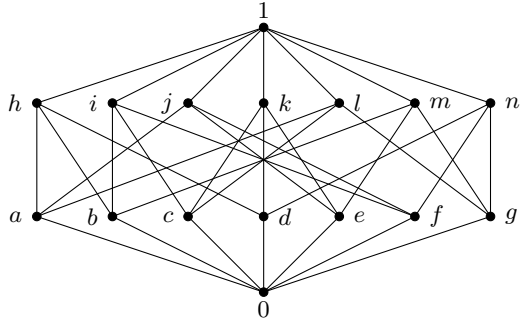
\begin{figure}[ht]
\centering
\begin{tikzpicture}[scale=1,radius=1.6pt,font=\small]
\filldraw (0,0) circle node [below] {$0$} coordinate (0);
\filldraw (0,3.5) circle node [above] {$1$} coordinate (1);
\filldraw (-3,1) circle node [left=2pt] {$a$} coordinate (a);
\filldraw (-2,1) circle node [left=2pt] {$b$} coordinate (b);
\filldraw (-1,1) circle node [left=2pt] {$c$} coordinate (c);
\filldraw (0,1) circle node [right=2pt] {$d$} coordinate (d);
\filldraw (1,1) circle node [right=2pt] {$e$} coordinate (e);
\filldraw (2,1) circle node [right=2pt] {$f$} coordinate (f);
\filldraw (3,1) circle node [right=2pt] {$g$} coordinate (g);
\filldraw (-3,2.5) circle node [left=2pt] {$h$} coordinate (h);
\filldraw (-2,2.5) circle node [left=2pt] {$i$} coordinate (i);
\filldraw (-1,2.5) circle node [left=2pt] {$j$} coordinate (j);
\filldraw (0,2.5) circle node [right=2pt] {$k$} coordinate (k);
\filldraw (1,2.5) circle node [right=2pt] {$l$} coordinate (l);
\filldraw (2,2.5) circle node [right=2pt] {$m$} coordinate (m);
\filldraw (3,2.5) circle node [right=2pt] {$n$} coordinate (n);
\draw (0) -- (a) -- (h) -- (1) -- (n) -- (g) -- (0) -- (1);
\draw (0) -- (b) -- (h) -- (d) -- (n) -- (f) -- (0);
\draw (0) -- (c) -- (i) -- (1) -- (m) -- (e) -- (0);
\draw (c) -- (k) -- (e);
\draw (a) -- (j) -- (e);
\draw (a) -- (l) -- (g);
\draw (b) -- (i) -- (f);
\draw (b) -- (m) -- (g);
\draw (1) -- (j) -- (f);
\draw (1) -- (l) -- (c);
\end{tikzpicture}
\caption{The lattice from Example \ref{EXA3}}\label{F4}
\end{figure}

\begin{example}\label{EXA3}
Consider the lattice in Figure~\ref{F4}.
This lattice is modular and simple; in fact, it is the subspace lattice of $V(\mathbb{Z}_2,3)$, the $3$-dimensional vector space over the field $\mathbb{Z}_2$, where
$a=\langle(1,0,0)\rangle$,
$b=\langle(1,1,0)\rangle$,
$c=\langle(1,0,1)\rangle$,
$d=\langle(0,1,0)\rangle$,
$e=\langle(1,1,1)\rangle$,
$f=\langle(0,1,1)\rangle$,
and
$g=\langle(0,0,1)\rangle$.
Up to isomorphism, there are two complementations making this lattice into a lattice with complementation that belongs to $\mathcal{W}$.
The first is involutive and is given by
\begin{center}
\begin{tabular}{c|cccccccccccccc}
$x$ & $a$ & $b$ & $c$ & $d$ & $e$ & $f$ & $g$ & $h$ & $i$ & $j$ & $k$ & $l$ & $m$ & $n$ \\
\hline
$x'$ & $k$ & $j$ & $m$ & $i$ & $n$ & $l$ & $h$ & $g$ & $d$ & $b$ & $a$ & $f$ & $c$ & $e$
\end{tabular}
\end{center}
The second is given by
\begin{center}
\begin{tabular}{c|cccccccccccccc}
$x$ & $a$ & $b$ & $c$ & $d$ & $e$ & $f$ & $g$ & $h$ & $i$ & $j$ & $k$ & $l$ & $m$ & $n$ \\
\hline
$x'$ & $m$ & $n$ & $h$ & $j$ & $i$ & $l$ & $k$ & $e$ & $a$ & $g$ & $f$ & $d$ & $c$ & $b$
\end{tabular}
\end{center}
In both cases, De Morgan's laws \eqref{DM} fail (e.g., $c\le k$ but $m\ngeq a$ and $h\ngeq f$).
These two lattices with complementation together with the Boolean algebra $\mathbf{2}^4$ are, up to isomorphism, the only members of $\mathcal{W}$ of size $16$.

As an aside, there exists another involutive complementation on the above lattice for which, however, the resulting lattice with complementation does not belong to $\mathcal{W}$ (cf.\ \cite[Example 22]{CL19}):
\begin{center}
\begin{tabular}{c|cccccccccccccc}
$x$ & $a$ & $b$ & $c$ & $d$ & $e$ & $f$ & $g$ & $h$ & $i$ & $j$ & $k$ & $l$ & $m$ & $n$ \\
\hline
$x'$ & $n$ & $k$ & $j$ & $l$ & $i$ & $m$ & $h$ & $g$ & $e$ & $c$ & $b$ & $d$ & $f$ & $a$
\end{tabular}
\end{center}
The failure of quasi-identity \eqref{W} is witnessed, e.g., by $g'\wedge c=h\wedge c=0$ and $g\wedge c'=g\wedge j=0$.
\end{example}

In lattices with complementation, we can define the so-called \emph{symmetric difference} $+$ as in Boolean algebras by
\[
x+y:=(x'\wedge y)\vee (x\wedge y').
\]
It is obvious that the identities 
\[
x+x\approx 0,\quad x+y\approx y+x,\quad x+0\approx x \quad\text{and}\quad x+1\approx x'
\]
are satisfied. However, the quasi-identity 
\[
\tag{$\mathrm{W}^+$}\label{equ1}
x+y\approx 0 \;\Rightarrow\; x\approx y,
\]
which is equivalent to \eqref{W} and certainly holds in Boolean algebras, is not satisfied by all lattices with complementation. For example, \eqref{equ1} fails in $\mathbf{M}_n'$ with $n\ge 4$ (see Example~\ref{EXA1} and Proposition~\ref{L:H2} below); it holds only in $\mathbf{M}_3'$, where the symmetric difference is given as follows:

\begin{center}
\begin{tabular}{c|ccccc}
$+$ & $0$ & $a_0$ & $a_1$ & $a_2$ & $1$ \\
\hline
$0$ & $0$ & $a_0$ & $a_1$ & $a_2$ & $1$ \\
$a_0$ & $a_0$ & $0$ & $a_1$ & $a_0$ & $a_1$ \\
$a_1$ & $a_1$ & $a_1$ & $0$ & $a_2$ & $a_2$ \\
$a_2$ & $a_2$ & $a_0$ & $a_2$ & $0$ & $a_0$ \\
$1$ & $1$ & $a_1$ & $a_2$ & $a_0$ & $0$
\end{tabular}
\end{center}

The first effect of \eqref{W} or \eqref{equ1} is nearly trivial: it forces complementation to be injective.

\begin{lemma}\label{L:H1}
If $\mathbf{L}\in\mathcal{W}$, then the complementation is injective.
\end{lemma}

\begin{proof}
Let $\mathbf{L}\in\mathcal{W}$. For any
$a,b\in L$, if $a'=b'$, then $a+b=(a'\wedge b)\vee(a\wedge b')=(b'\wedge b)\vee(a\wedge a')=0$, which implies $a=b$ by \eqref{equ1}.
\end{proof}

As mentioned in Example~\ref{EXA1}, we now show that among the lattices $\mathbf{M}_\kappa$ (with possibly infinite $\kappa$), only $\mathbf{M}_2$ and $\mathbf{M}_3$ can be equipped with a complementation satisfying \eqref{W}, or equivalently \eqref{equ1}:

\begin{proposition}\label{L:H2}
If $\kappa\ge 4$, then the lattice $\mathbf{M}_\kappa$ does not admit any complementation such that the resulting lattice with complementation belongs to $\mathcal{W}$.
\end{proposition}

\begin{proof}
Suppose that $'$ is a complementation on $\mathbf{M}_\kappa$ such that the resulting lattice with complementation belongs to $\mathcal{W}$.
Let $A:=M_\kappa\setminus\{0,1\}$.
Clearly, $x'\in A\setminus\{x\}$ for all $x\in A$. Let $a\in A$ and put $b:=a'$. 

Case 1. $b'=a$.
Let $c\in A\setminus\{a,b\}$ and put $d:=c'$. Since the elements $a,b,c$ are pairwise distinct and $'$ is injective by Lemma~\ref{L:H1}, the elements $b,a,d$ are also pairwise distinct and 
$a+c=(a'\wedge c)\vee(a\wedge c')=(b\wedge c)\vee(a\wedge d)=0$, contradicting $a\ne c$.

Case 2. $b'\ne a$. 
Put $c:=b'$. Then $a\ne c\ne b$ and hence $a,b,c$ are pairwise distinct. 

Case 2a. $c'=a$. 
Let $d\in A\setminus\{a, b, c\}$ and put $e:=d'$. Since $a,b,c,d$ are pairwise distinct, the same is true for $b,c,a,e$ and we have $a+d=(a'\wedge d)\vee(a\wedge d')=(b\wedge d)\vee(a\wedge e)=0$, contradicting $a\ne d$. 

Case 2b. $c'\ne a$. 
Put $d:=c'$. Then $d\ne a$ and $a+c=(a'\wedge c)\vee(a\wedge c')=(b\wedge c)\vee(a\wedge d)=0$, again contradicting $a\ne c$.
\end{proof}

Somewhat related to \eqref{W} or \eqref{equ1} is the quasi-identity
\begin{equation}\label{equ2}
x\wedge y'\approx 0 \;\Rightarrow\; x\le y,
\end{equation}
which is, however, too strong:

\begin{lemma}
A lattice with complementation satisfies quasi-identity \eqref{equ2} if and only if it is a Boolean algebra.
\end{lemma}

\begin{proof}
Let $\mathbf{L}$ be a lattice with complementation.
It is obvious that $\mathbf{L}$ satisfies \eqref{equ2} when it is a Boolean algebra. Conversely, suppose that $\mathbf{L}$ satisfies \eqref{equ2}. Then, for every $a\in L$, $a''\wedge a'=0$ implies $a''\le a$, and any of the following statements implies the next one:
\[
a'''\wedge a''=0;\quad a'''\le a';\quad 
a\wedge a'''\le a\wedge a'=0;\quad a\le a''.
\]
This shows $a''=a$. Moreover, for every $a,b\in L$, any of the following statements implies the next one:
\[
a\le b;\quad a\wedge b'=0;\quad b'\wedge a''=0;\quad b'\le a'.
\]
Hence $'$ is an antitone involution, and so $\mathbf{L}$ is an ortholattice.
At this point, it suffices to show that $\mathbf{L}$ is uniquely complemented, since it is known that a uniquely complemented lattice satisfying De Morgan's laws is distributive (e.g.\ \cite[Chapter 6]{Bl}). 
To this end, suppose that $b\in L$ is a complement of $a\in L$. Then $b\wedge a''=b\wedge a=0$ yields $b\le a'$ by \eqref{equ2}, and $a\vee b=1$ is equivalent to $a'\wedge b'=0$, which yields $a'\le b$ also by \eqref{equ2}. Thus $b=a'$.
\end{proof}

\begin{lemma}
Let $\mathbf{L}$ be a lattice with complementation. 
Consider the following properties:
\begin{enumerate}[\rm(i)]
\item
$\mathbf{L}$ does not have a pentagon as a $\{0,1\}$-sublattice;
\item
$\mathbf{L}$ satisfies the quasi-identity
$(x\le y \;\mathrel{\&}\; x'\approx y') \;\Rightarrow\; x\approx y$;
\item
$\mathbf{L}$ does not have a subalgebra whose lattice reduct is a pentagon.
\end{enumerate}
Then {\rm(i)} $\Rightarrow$ {\rm(ii)} $\Rightarrow$ {\rm(iii)}.
If $\mathbf{L}\in\mathcal{M}$, then $\mathbf{L}$ satisfies {\rm(i)}. If $\mathbf{L}\in\mathcal{W}$, then $\mathbf{L}$ satisfies {\rm(ii)}. 
If $\mathbf{L}$ satisfies De Morgan's laws, then condition {\rm(ii)} is equivalent to the injectivity of the complementation. Hence, for $\mathbf{L}\in\mathcal{M}_{DM}$, the complementation is injective (cf.\ Lemma~\ref{lem3}).
\end{lemma}

\begin{proof}
(i) implies (ii) because $a<b$ and $a'=b'$ would mean that the elements $0,a,b,a',1$ form a $\{0,1\}$-sublattice isomorphic to a pentagon. 
(ii) implies (iii) because, clearly, a sub\-algebra whose lattice reduct is a pentagon does not satisfy the quasi-identity in (ii).

If $\mathbf{L}\in\mathcal{M}$, then it does not have a pentagon as a sublattice. If $\mathbf{L}\in\mathcal{W}$, then it certainly satisfies (ii) because the complementation is injective by Proposition \ref{L:H1}. Finally, if $\mathbf{L}$ satisfies De Morgan's laws and (ii), then for any $a,b\in L$ with $a'=b'$ we have $(a\wedge b)'=a'\vee b'=a'\wedge b'=(a\vee b)'$, whence $a\wedge b=a\vee b$ by (ii), and so $a=b$.
\end{proof}

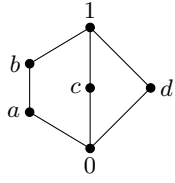
\begin{figure}[ht]
\centering
\begin{tikzpicture}[scale=0.8,radius=2pt,font=\small]
\filldraw (1,0) circle node [below] {$0$} coordinate (0);
\filldraw (0,0.6) circle node [left] {$a$} coordinate (a);
\filldraw (0,1.4) circle node [left] {$b$} coordinate (b);
\filldraw (2,1) circle node [right] {$d$} coordinate (d);
\filldraw (1,1) circle node [left] {$c$} coordinate (c);
\filldraw (1,2) circle node [above] {$1$} coordinate (1);
\draw (0) -- (a) -- (b) -- (1) -- (d) -- (0);
\draw (0) -- (1);
\end{tikzpicture}
\caption{The lattice from Example \ref{EXA28}}\label{F5}
\end{figure}

\begin{example}\label{EXA28}
Conditions (i) and (ii) are not equivalent.
Indeed, consider the lattice in Figure~\ref{F5} with complementation $a\leftrightarrow c$ and $b\leftrightarrow d$. Clearly, (i) is not satisfied, whereas (ii) is, since no two different comparable elements have the same complement.
Note that De Morgan's laws fail, as $a<b$ but $a'=c\ngeq d=b'$.

Likewise, conditions (ii) and (iii) are not equivalent.
Again, consider the above lattice with complementation $a\mapsto c$, $b\mapsto c$ and $c\leftrightarrow d$. Then (ii) is not satisfied, as $a<b$ and $a'=b'$. However, (iii) is satisfied. 
De Morgan's laws fail again, as $(a\vee d)'=0$ but $a'\wedge d'=c$.
\end{example}

\section{A construction}
\label{SEC3}

In order to avoid possible discrepancies in the signatures of the involved algebras, we emphasize that we treat Boolean lattices as bounded lattices rather than as lattices with complementation. (After all, we carefully distinguish Boolean lattices and Boolean algebras throughout the paper.)

The following result describes all complementations on the direct product $\mathbf{K}\times\mathbf{B}$, where $\mathbf{K}$ is a complemented lattice and $\mathbf{B}$ is a Boolean lattice.
In brief, it suffices to associate a complementation on $\mathbf{K}$ with each element of the Boolean lattice. We do not need to exclude the trivial Boolean lattice, since the lattice with complementation constructed in Proposition \ref{L:V1} is then just a copy of $\mathbf{K}$ equipped with the chosen complementation.

\begin{proposition}\label{L:V1}
Let $\mathbf{K}$ be a complemented lattice and $\mathbf{B}$ a Boolean lattice.
\begin{enumerate}[\rm(i)]
\item
Let $S$ be an ``action of $\mathbf{B}$ on $\mathbf{K}$ by (injective) complementations'', i.e., a mapping $x\mapsto S_x$ from the Boolean lattice to the set of (injective) complementations on the complemented lattice. 
If we define
\[
C_S(a,x):=(S_x(a),x')
\quad\text{for } (a,x)\in K\times B,
\]
then
\[
\mathfrak{L}(\mathbf{K},\mathbf{B},S)=(K\times B,\vee,\wedge,C_S,(0,0),(1,1))
\]
is a lattice with (injective) complementation.
\item
Every (injective) complementation on the direct product $\mathbf{K}\times\mathbf{B}$ is of this form for some ``action'' $S$.
\item
In particular, if $\mathbf{K}=\mathbf{M}_n$, then $\mathfrak{L}(\mathbf{M}_n,\mathbf{B},S)$ belongs to $\mathcal{W}$ if and only if 
{\rm(a)}~$n=2$, in which case $\mathfrak{L}(\mathbf{M}_2,\mathbf{B},S)$ is a Boolean algebra, or
{\rm(b)}~$n=3$ and the complementations $S_x$ on $\mathbf{M}_3$ are injective. 
\end{enumerate}
\end{proposition}

\begin{proof}
(i) For every $(a,x)\in K\times B$, we have
\begin{gather*}
(a,x)\vee C_S(a,x)=(a\vee S_x(a),x\vee x')=(1,1),\\
(a,x)\wedge C_S(a,x)=(a\wedge S_x(a),x\wedge x')=(0,0).
\end{gather*}
Thus $C_S$ is a complementation on the direct product $\mathbf{K}\times\mathbf{B}$.
Suppose that the complementations $S_x$ are injective. If $C_S(a,x)=C_S(b,y)$ for some $(a,x)$, $(b,y)\in K\times B$, then $x'=y'$ in $\mathbf{B}$ implies $x=y$, and since the complementation $S_x=S_y$ on $\mathbf{K}$ is injective, $S_x(a)=S_y(b)$ implies $a=b$. 
Thus the complementation $C_S$ on $\mathbf{K}\times\mathbf{B}$ is injective. 

(ii) Let $'$ be a complementation on $\mathbf{K}\times\mathbf{B}$. For every $x\in B$ we define $S_x$ by stipulating
\begin{equation}\label{E:Sx}
S_x(a):=a^* \quad\text{if } (a,x)'=(a^*,x'),
\end{equation}
i.e., $(a,x)'=(S_x(a),x')$.
Then $S_x$ is a complementation on $\mathbf{K}$. Indeed, for every $a\in K$, we have
$a\vee S_x(a)=1$ and $a\wedge S_x(a)=0$,
because $(1,1)=(a,x)\vee (a,x)'=(a\vee S_x(a),x\vee x')$
and $(0,0)=(a\wedge S_x(a),x\wedge x')$.
Moreover, if $'$ is injective, then so is $S_x$, since for any $a,b\in K$, $S_x(a)=S_x(b)$ yields
$(a,x)'=(S_x(a),x')=(S_x(b),x')=(b,x)'$, whence $a=b$.

Finally, it is obvious that the complementation $C_S$ induced by the ``action'' $S$ defined by \eqref{E:Sx} is precisely the original complementation $'$ because $C_S(a,x)=(S_x(a),x')=(a,x)'$ for all $(a,x)\in K\times B$.

(iii) Let $\mathfrak{L}(\mathbf{M}_n,\mathbf{B},S)\in\mathcal{W}$. 
For any $x\in B$, the complementation $S_x$, or more accurately, the lattice $\mathbf{M}_n$ equipped with $S_x$, satisfies quasi-identity \eqref{W}. 
Indeed, if $S_x(a)\wedge b=0$ and $a\wedge S_x(b)=0$ for some $a,b\in M_n$, then in $\mathfrak{L}(\mathbf{M}_n,\mathbf{B},S)$ we have $C_S(a,x)\wedge (b,x)=(S_x(a)\wedge b,x'\wedge x)=(0,0)$ and $(a,x)\wedge C_S(b,x)=(a\wedge S_x(b),x\wedge x')=(0,0)$, whence $(a,x)=(b,x)$, i.e. $a=b$.
Moreover, $S_x$ is injective, because if $S_x(a)=S_x(b)$ for some $a,b\in M_n$, then $S_x(a)\wedge b=0$ and $a\wedge S_x(b)=0$, which yields $a=b$ by \eqref{W} (cf.\ Lemma~\ref{L:H1}).

Then $n=3$, because for $n\ge 4$ there exist $a,b\in M_n$ such that $a,b,S_x(a)$ and $S_x(b)$ are four distinct elements with $S_x(a)\wedge b=0$ and $a\wedge S_x(b)=0$, contradicting \eqref{W}. This can be seen as follows.
The complementation $S_x$ (restricted to $M_n\setminus\{0,1\}=\{a_0,\dots,a_{n-1}\}$) is either a single cycle of length $n$ or a product of several disjoint cycles of length $\ge 2$.
In the former case, we may take any $a\in M_n\setminus\{0,1\}$ and $b=S_x(S_x(a))$. In the latter case, we take $a$ and $b$ from the supports of two disjoint cycles. 
(Cf.\ Example \ref{EXA1} and Proposition \ref{L:H2}.) 

Conversely, let $n=3$ and suppose that the complementations $S_x$ are injective. Then each $S_x$ is (determined by) either $(a_0\;a_1\;a_2)$ or $(a_2\;a_1\;a_0)$, and hence $\mathbf{M}_3$ equipped with $S_x$ is isomorphic to $\mathbf{M}_3'$, which satisfies \eqref{W}.
Hence, if $C_S(a,x)\wedge (b,y)=(S_x(a)\wedge b,x'\wedge y)=(0,0)$ and
$(a,x)\wedge C_S(b,y)=(a\wedge S_y(b),x\wedge y')=(0,0)$ 
for some $(a,x),(b,y)\in M_3\times B$,
then $x=y$ and $S_x(a)\wedge b=0$ and $a\wedge S_x(b)=0$, whence $a=b$, 
proving $(a,x)=(b,y)$. Thus $\mathfrak{L}(\mathbf{M}_3,\mathbf{B},S)\in\mathcal{W}$.
\end{proof}

Now, let $\mathbf{K}=\mathbf{M}_n$ and
$\mathbf{B}=\mathbf{2}^k$ for some integers $n\ge 3$ and $k\ge 1$.
Let $\theta_0$ be the congruence induced by the projection $\psi_0\colon (a,x)\mapsto x$, and for every $i\in\{1,\dots,k\}$, let $\theta_i$ be the \emph{lattice} congruence induced by the projection $\psi_i$ deleting the $i$-th coordinate of $x=(x_1,\dots,x_k)\in B$, i.e., 
$\psi_i\colon (a,x_1,\dots,x_i,\dots,x_k)\mapsto (a,x_1,\dots,x_{i-1},x_{i+1},\dots,x_k)$.

Note that $\theta_0$ is indeed a congruence on the \emph{lattice with complementation} $\mathfrak{L}(\mathbf{M}_n,\mathbf{B},S)$, for every ``action'' $S$, since $\psi_0$ preserves complementation,
whereas $\theta_1,\dots,\theta_k$ are, in general, only congruences on the \emph{lattice} $\mathbf{M}_n\times\mathbf{B}$.

It is easy to see that, once a certain complementation on $\mathbf{M}_n$ is chosen, the complementation $C_S$ on $\mathbf{M}_n\times\mathbf{B}$, as defined in Proposition~\ref{L:V1}, is pointwise if and only if the ``action'' $S$ assigns this chosen complementation to every $x\in B$.

\begin{lemma}\label{L:V2}
Let $\mathfrak{L}(\mathbf{M}_n,\mathbf{B},S)$ be as above, with finite $\mathbf{B}=\mathbf{2}^k$. 
The lattice with complementation $\mathfrak{L}(\mathbf{M}_n,\mathbf{B},S)$ is sub\-directly irreducible if and only if none of the lattice congruences $\theta_1,\dots,\theta_k$ is a congruence on $\mathfrak{L}(\mathbf{M}_n,\mathbf{B},S)$.
\end{lemma}

\begin{proof}
For simplicity, we write $\mathbf{L}$ instead of $\mathfrak{L}(\mathbf{M}_n,\mathbf{B},S)$. 
The congruence lattice of the direct product $\mathbf{M}_n\times\mathbf{B}$, the lattice reduct of $\mathbf{L}$, is a finite Boolean lattice of size $2^{k+1}$, whose atoms are precisely $\theta_0,\theta_1,\dots,\theta_k$.
Hence, if $\mathbf{L}$ is subdirectly irreducible, then, recalling that $\theta_0$ is a congruence on $\mathbf{L}$, we see that none of $\theta_1,\dots,\theta_k$ is a congruence on $\mathbf{L}$.

Conversely, suppose that none of $\theta_1,\dots,\theta_k$ is a congruence on $\mathbf{L}$. The goal is to show that $\theta_0$ is the monolith of the congruence lattice of $\mathbf{L}$.
The case $k=1$ is clear, so let $k\ge 2$.
Every nonzero congruence $\zeta$ on $\mathbf{L}$ distinct from $\theta_0$, being a lattice congruence, is the join of some $\theta_i$'s distinct from $\theta_0$, say 
$\zeta=\theta_{i_1}\vee\dots\vee\theta_{i_m}$ for some 
$\{i_1,\dots,i_m\}\subseteq\{1,\dots,k\}$.
Since congruences of complemented modular lattices permute, we have
$\zeta=\theta_{i_1}\circ\dots\circ\theta_{i_m}$.

Since $\theta_{i_1}$ is not compatible with $C_S$, there exist $(a,x)\ne (b,y)$ such that $((a,x),(b,y))\in\theta_{i_1}$ but $(C_S(a,x),C_S(b,y))\notin\theta_{i_1}$.
This means that $(a,x)$ and $(b,y)$ differ only in the $i_1$-th coordinate of $x$ and $y$, i.e., 
$a=b$, $x_{i_1}\ne y_{i_1}$ and $x_j=y_j$ for all $j\ne i_1$. 
On the other hand,
$C_S(a,x)=(S_x(a),x')$ and $C_S(b,y)=(S_y(a),y')$ differ in more than just the $i_1$-th coordinate of $x'$ and $y'$. Since $x_j'=y_j'$ for all $j\ne i_1$, this is only possible if $S_x(a)\ne S_y(a)$.

However, $\zeta=\theta_{i_1}\circ\dots\circ\theta_{i_m}$ is a congruence on $\mathbf{L}$ and $((a,x),(b,y))\in\zeta$ with $a=b$. 
Hence $(C_S(a,x),C_S(a,y))\in\zeta$ and we have
\[
(S_x(a),x')\equiv_{\theta_{i_1}}(c_1,z_1)
\equiv_{\theta_{i_2}}\cdots\equiv_{\theta_{i_{m-1}}}
(c_{m-1},z_{m-1})\equiv_{\theta_{i_m}} (S_y(a),y')
\]
for some $(c_1,z_1),\dots,(c_{m-1},z_{m-1})$.
Thus $S_x(a)=c_1=\dots=c_{m-1}=S_y(a)$, contradicting $S_x(a)\ne S_y(a)$.

Therefore, every nonzero congruence of $\mathbf{L}$ exceeds $\theta_0$, as desired.
\end{proof}

\begin{lemma}\label{L:V3}
Let $\mathfrak{L}(\mathbf{M}_n,\mathbf{B},S)$ be as above, with finite $\mathbf{B}=\mathbf{2}^k$. 
Suppose that there exists $x\in B$ such that for every $y\in B\setminus\{x\}$ there exists $a\in M_n$ with $S_x(a)\ne S_y(a)$. Then $\mathfrak{L}(\mathbf{M}_n,\mathbf{B},S)$ is subdirectly irreducible.
\end{lemma}

\begin{proof}
Let $x=(x_1,\dots,x_k)\in B$ be a $k$-tuple satisfying the condition above. 
For a fixed but arbitrary $i\in\{1,\dots,k\}$, let $y\in B$ be the $k$-tuple obtained from $x$ by replacing $x_i$ with $x_i'$, i.e., $y_i=x_i'$ and $y_j=x_j$ otherwise. 
Clearly, $x\ne y$ and so there exists $a\in M_n$ such that $S_x(a)\ne S_y(a)$.
We then have $((a,x),(a,y))\in\theta_i$, while
$(C_S(a,x),C_S(a,y))\notin\theta_i$, since
$C_S(a,x)=(S_x(a),x')$ and $C_S(a,y)=(S_y(a),y')$
where $S_x(a)\ne S_y(a)$.
Thus $\theta_i$ is not a congruence. By Lemma~\ref{L:V2} we get that $\mathfrak{L}(\mathbf{M}_n,\mathbf{B},S)$ is subdirectly irreducible.
\end{proof}

In particular, in the case of $\mathbf{M}_3$ and $\mathbf{B}=\mathbf{2}^k$, we may pick any $x\in B$ and define $S_x$ to be the complementation determined by the cycle $(a_0\;a_1\;a_2)$, and $S_y$ to be the complementation determined by the inverse cycle $(a_2\;a_1\;a_0)$ for every $y\in B\setminus\{x\}$.
We then have $S_x(a_0)\ne S_y(a_0)$, and so $\mathfrak{L}(\mathbf{M}_3,\mathbf{B},S)$ is subdirectly irreducible by Lemma~\ref{L:V3}. Therefore, from Proposition~\ref{L:V1} and Lemma~\ref{L:V3} we obtain:

\begin{corollary}\label{C:34}
\leavevmode
\begin{enumerate}[\rm(i)]
\item
For every finite Boolean lattice $\mathbf{B}$, there exists a complementation on the bounded lattice $\mathbf{M}_3\times\mathbf{B}$ such that the resulting lattice with complementation is subdirectly irreducible and belongs to the class $\mathcal{W}$.
\item
There are infinitely many finite subdirectly irreducible algebras in $\mathcal{W}$.
\end{enumerate}
\end{corollary}

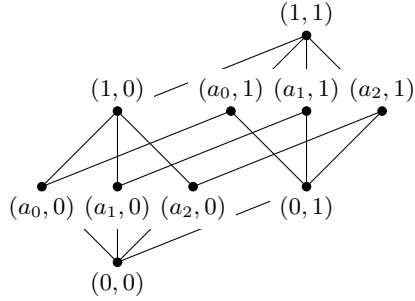
\begin{figure}[ht]
\centering
\begin{tikzpicture}[scale=1,radius=1.6pt,font=\small]
\coordinate (0) at (1,0);
\coordinate (a) at (0,1);
\coordinate (b) at (1,1);
\coordinate (c) at (2,1);
\coordinate (d) at (1,2);
\coordinate (e) at (3.5,1);
\coordinate (f) at (2.5,2);
\coordinate (g) at (3.5,2);
\coordinate (h) at (4.5,2);
\coordinate (1) at (3.5,3);
\draw (0) -- (e);
\draw (a) -- (f); 
\draw (b) -- (g);
\draw (c) -- (h);
\draw (d) -- (1);
\draw (0) -- (b) -- (d) -- (a) -- (0) -- (c) -- (d);
\draw (e) -- (g) -- (1) -- (f) -- (e) -- (h) -- (1);
\draw (0) node [below] {$(0,0)$};
\draw (a) node [fill=white,below] {$(a_0,0)$};
\draw (b) node [fill=white,below] {$(a_1,0)$};
\draw (c) node [fill=white,below] {$(a_2,0)$};
\draw (d) node [fill=white,above] {$(1,0)$};
\draw (e) node [fill=white,below] {$(0,1)$};
\draw (f) node [fill=white,above] {$(a_0,1)$};
\draw (g) node [fill=white,above] {$(a_1,1)$};
\draw (h) node [fill=white,above] {$(a_2,1)$};
\draw (1) node [above] {$(1,1)$};
\filldraw (0) circle;
\filldraw (a) circle;
\filldraw (b) circle;
\filldraw (c) circle;
\filldraw (d) circle;
\filldraw (e) circle;
\filldraw (f) circle;
\filldraw (g) circle;
\filldraw (h) circle;
\filldraw (1) circle;
\end{tikzpicture}
\caption{The lattice from Example \ref{EXA4}}\label{F6}
\end{figure}

\begin{example}\label{EXA4}
The simplest nontrivial example of the above construction is $\mathfrak{L}(\mathbf{M}_3,\mathbf{2},S)$, i.e., $n=3$, $k=1$ and the lattice reduct is $\mathbf{M}_3\times\mathbf{2}$ (see Figure~\ref{F6}).
In principle, there are two ways of defining the complementations $S_0$ and $S_1$ (two ``actions'' of $\mathbf{2}$ on $\mathbf{M}_3$) that give rise to lattices with complementation isomorphic to those from Example~\ref{EXA2}:

If $S_0$ and $S_1$ coincide, we define this single complementation to be the one determined by the cycle $(a_0\;a_1\;a_2)$ (choosing $(a_2\;a_1\;a_0)$ would yield the same result, up to isomorphism), and we have:
\begin{center}
\begin{tabular}{c|cccccccc}
$(a,x)$ & $(a_0,0)$ & $(a_1,0)$ & $(a_2,0)$ & $(1,0)$
& $(0,1)$ & $(a_0,1)$ & $(a_1,1)$ & $(a_2,1)$ \\ 
\hline
$C_S(a,x)$ & $(a_1,1)$ & $(a_2,1)$ & $(a_0,1)$ & $(0,1)$
& $(1,0)$ & $(a_1,0)$ & $(a_2,0)$ & $(a_0,0)$
\end{tabular}
\end{center}
This $\mathfrak{L}(\mathbf{M}_3,\mathbf{2},S)$ is isomorphic to the direct product $\mathbf{M}_3'\times\mathbf{2}$ (it is isomorphic to $\mathbf{H}_1$ from Example~\ref{EXA2}).

If $S_0$ and $S_1$ are distinct, we define $S_0$ and $S_1$ to be the complementations determined respectively by $(a_0\;a_1\;a_2)$ and $(a_2\;a_1\;a_0)$ (or vice versa, but that choice would yield the same result, up to isomorphism). We then have
\begin{center}
\begin{tabular}{c|cccccccc}
$(a,x)$ & $(a_0,0)$ & $(a_1,0)$ & $(a_2,0)$ & $(1,0)$
& $(0,1)$ & $(a_0,1)$ & $(a_1,1)$ & $(a_2,1)$ \\
\hline
$C_S(a,x)$ & $(a_1,1)$ & $(a_2,1)$ & $(a_0,1)$ & $(0,1)$
& $(1,0)$ & $(a_2,0)$ & $(a_0,0)$ & $(a_1,0)$
\end{tabular}
\end{center}
This $\mathfrak{L}(\mathbf{M}_3,\mathbf{2},S)$ is subdirectly irreducible (it is isomorphic to $\mathbf{H}_2$ from Example~\ref{EXA2}).
\end{example}

Since $\mathbf{M}_n'$ has no proper subalgebras, we may expect that the subalgebras of the direct product of $\mathbf{M}_n'$ with a Boolean algebra can be described in terms of certain subsets of this Boolean algebra, as stated in the following lemma.

Given a lattice with complementation $\mathbf{L}$, for any  $X,Y\subseteq L$ we define
\begin{align*}
X' & :=\{x'\mid x\in X\}, \\
X\vee Y & :=\{x\vee y\mid x\in X,y\in Y\}, \\	
X\wedge Y & :=\{x\wedge y\mid x\in X,y\in Y\}.
\end{align*}

\begin{lemma}
Let $\mathbf{B}$ be a Boolean algebra and 
$\emptyset\neq Z\subseteq M_n\times B$ with $n\ge 3$. The following are equivalent:
\begin{enumerate}[\rm(i)]
\item 
$Z$ is a subuniverse of the direct product 
$\mathbf{M}_n'\times\mathbf{B}$;
\item 
there exists a subuniverse $X$ of the $\{0\}$-lattice reduct $(B,\vee,\wedge,0)$ and a (possibly empty) subuniverse $Y$ of the Boolean algebra $\mathbf{B}$ such that $Y\subseteq X\cap X'$, $X\vee Y\subseteq Y$, $X'\wedge Y\subseteq Y$, $X\vee X'\subseteq X'$, $X\wedge X'\subseteq X$, and
\[
Z=(\{0\}\times X)\cup(\{a_0,\dots,a_{n-1}\}\times Y)\cup(\{1\}\times X').
\]
\end{enumerate}
\end{lemma}

\begin{proof}
(i) $\Rightarrow$ (ii): 
Let $Z$ be a nonempty subuniverse of the direct product 
$\mathbf{M}_n'\times\mathbf{B}$ and 
$A:=M_n\setminus\{0,1\}=\{a_0,\dots,a_{n-1}\}$. Further, put 
\[
X:=\{x\in B\mid(0,x)\in Z\}
\quad\text{and}\quad
Y:=\{x\in B\mid(a_i,x)\in Z \text{ for some } a_i\}.
\] 
Note that $X'=\{x\in B\mid (1,x)\in Z\}$.
It is clear that $X$ is a subuniverse of $(B,\vee,\wedge,0)$, and $X'$ of $(B,\vee,\wedge,1)$.
Also, $Y$ is a subuniverse of $(B,\vee,\wedge)$. 
Let $x\in Y$, i.e., $(a_i,x)\in Z$ for some $a_i$.
Then $(a_j,x)\in Z$ for every $a_j$ in the ``even orbit'' of $a_i$, $\{a_i,a_i'',\dots\}=\{a_{i+k}\mid k\text{ even}\}$, whence $(0,x),(1,x)\in Z$.
Note that this ``orbit'' is all of $A$ when $n$ is odd, while it is only half of $A$ when $n$ is even.
In the latter case, $(a_j,x')\in Z$ for every $a_j$ in the ``odd orbit'' of $a_i$, $\{a_i',a_i''',\dots\}=\{a_{i+k}\mid k\text{ odd}\}$, and $(0,x'),(1,x')\in Z$.
Now, since $(a_{i+1},x'),(0,x)\in Z$, we have 
$(a_{i+1},1)\in Z$. But we also have $(1,x)\in Z$, and so $(a_{i+1},x)\in Z$. Therefore, $(a_j,x)\in Z$ for every $a_j$ in the ``even orbit'' of $a_{i+1}$, which is just the ``odd orbit'' of $a_i$, and hence $(a_j,x)\in Z$ for every $a_j\in A$.
Analogously, $(a_j,x')\in Z$ for every $a_j\in A$.
Thus $M_n\times\{x,x'\}\subseteq Z$, proving that $Y$ is a subuniverse of $\mathbf{B}$.

Now, if $x\in Y$, then $(a_i,x),(a_{i+1},x)\in Z$ for some (in fact, every) $a_i$, and hence $(0,x),(1,x)\in Z$, showing $x\in X\cap X'$. The other conditions, i.e., $X\vee Y\subseteq Y$, $X'\wedge Y\subseteq Y$, $X\vee X'\subseteq X'$ and $X\wedge X'\subseteq X$, are easily verified.

Finally, we have $(\{0\}\times X)\cup (\{1\}\times X')\subseteq Z$, and $A\times Y\subseteq Z$ because, by the above arguments, $x\in Y$ is equivalent to $A\times\{x,x'\}\subseteq Z$.
On the other hand, if $(a_i,x)\in Z$, then $x\in Y$, and so $(a_i,x)\in A\times Y$. Consequently, $Z$ equals 
$(\{0\}\times X)\cup (\{1\}\times X')\cup (A\times Y)$ as claimed.

(ii) $\Rightarrow$ (i): 
This follows immediately by considering all possible cases.
\end{proof}

\begin{example}
For $\mathbf{M}_3'$, if $\mathbf{B}$ is the two-element Boolean algebra, we have the following cases:
\begin{center}
\begin{tabular}{l|l|l}
$X$ & $Y$ & $Z$ \\
\hline
$\{0\}$ & $\emptyset$ & $\{(0,0),(1,1)\}$ \\
$B$ & $\emptyset$ & $B\times B$ \\
$B$ & $B$ & $M_3\times B$
\end{tabular}
\end{center}
If $\mathbf{B}$ is the four-element Boolean algebra with atoms $e$ and $f$, then we have the following cases:
\begin{center}
\begin{tabular}{l|l|l}
$X$ & $Y$ & $Z$ \\
\hline
$\{0\}$ & $\emptyset$ & $\{(0,0),(1,1)\}$ \\
$\{0,e\}$ & $\emptyset$ & $(\{0\}\times\{0,e\})\cup(\{1\}\times\{f,1\})$ \\
$\{0,f\}$ & $\emptyset$ & $(\{0\}\times\{0,f\})\cup(\{1\}\times\{e,1\})$ \\
$\{0,1\}$ & $\emptyset$ & $\{0,1\}\times\{0,1\}$ \\
$\{0,1\}$ & $\{0,1\}$   & $M_3\times\{0,1\}$ \\
$B$ & $\emptyset$ & $\{0,1\}\times B$ \\
$B$ & $B$ & $M_3\times B$
\end{tabular}
\end{center}
\end{example}

\section{Properties of $\mathcal{W}_{DM}$ and $\mathcal{M}_{DM}$; neutral elements}
\label{SEC4}

In this section, we describe some elementary properties of members of the classes $\mathcal{W}_{DM}$ and $\mathcal{M}_{DM}$, with an emphasis on the role of De Morgan's laws \eqref{DM}. We also characterize neutral elements in terms of complementation.

Following \cite{CL18}, we say that a lattice with complementation is \emph{weakly orthomodular} if it satisfies the quasi-identity
\begin{equation}\label{woml}
x\le y \;\Rightarrow\; x\vee (x'\wedge y)\approx y,
\end{equation}
and \emph{dually weakly orthomodular} if it satisfies the dual quasi-identity
\begin{equation}\label{dwoml}
x\ge y \;\Rightarrow\; x\wedge (x'\vee y)\approx y.
\end{equation}

While it is clear that \eqref{woml} can be equivalently expressed by either of the identities
\begin{gather*}
(x\wedge y)\vee\big((x\wedge y)'\wedge y\big)\approx y,\\
x\vee\big(x'\wedge (x\vee y)\big)\approx x\vee y,
\end{gather*}
perhaps slightly less obvious is that \eqref{woml} is also equivalent to the identity
\begin{equation}\label{woml-}
x\vee\big((x\wedge y)'\wedge y\big)\approx x\vee y.
\end{equation}
Indeed, \eqref{woml-} is obtained from the first identity above by taking the join with $x$, and conversely, \eqref{woml-} clearly implies \eqref{woml}. 

Dually, \eqref{dwoml} is equivalent to any of the following identities:
\begin{gather}
\notag
(x\vee y)\wedge\big((x\vee y)'\vee y\big)\approx y,\\
\notag
x\wedge\big(x'\vee (x\wedge y)\big)\approx x\wedge y,\\
\label{dwoml-}
x\wedge\big((x\vee y)'\vee y\big)\approx x\wedge y.
\end{gather}

We let $\mathcal{O}$ and $\mathcal{O}^d$ denote respectively the variety of weakly orthomodular lattices with complementation and the variety of dually weakly orthomodular lattices with complementation.
It is obvious that $\mathcal{M}\subseteq\mathcal{O}\cap\mathcal{O}^d$.

\begin{lemma}\label{lem3}
If $\mathbf{L}\in\mathcal{W}\cup (\mathcal{O}\cup \mathcal{O}^d)_{DM}$, then the complementation is injective.
\end{lemma}

\begin{proof}
By Lemma \ref{L:H1}, we know  that this is true for $\mathbf{L}\in\mathcal{W}$, so let $\mathbf{L}\in (\mathcal{O}\cup \mathcal{O}^d)_{DM}$.
For any $a,b\in L$, if $a'=b'$, then $(a\vee b)'=(a\wedge b)'$ and we have
$a\wedge b=(a\wedge b)\vee ((a\wedge b)'\wedge (a\vee b))=a\vee b$ 
by \eqref{woml} when $\mathbf{L}\in\mathcal{O}$, and 
$a\vee b=(a\vee b)\wedge ((a\vee b)'\vee (a\wedge b))=a\wedge b$
by \eqref{dwoml} when $\mathbf{L}\in\mathcal{O}^d$. 
In either case, $a=b$.
\end{proof}

\begin{lemma}\label{lem3a}
We have $\mathcal{W}_{DM}\cup\mathcal{M}\subseteq\mathcal{O}\cap \mathcal{O}^d$.
Hence, if $\mathbf{L}\in\mathcal{M}_{DM}$, then the complementation is injective.
\end{lemma}

However, it should be noted that in a modular lattice with complementation that does not satisfy De Morgan's laws, the complementation need not be injective. A simple example, already mentioned in Section~\ref{SEC1}, is the lattice $\mathbf{M}_3$ with complementation $a_0\mapsto a_1$ and $a_1\leftrightarrow a_2$.
On the other hand, there exist modular lattices with injective (or even involutive) complementation not satisfying De Morgan's laws (see Examples~\ref{EXA2}, \ref{EXA3} or \ref{EXA4}).

\begin{proof}
Let $\mathbf{L}\in\mathcal{W}_{DM}$. If $a\le b$ for some $a,b\in L$, then 
$(a\vee (a'\wedge b))\wedge b'\le (b\vee (a'\wedge b))\wedge b'=b\wedge b'=0$. Also,
$(a\vee (a'\wedge b))'\wedge b=a'\wedge (a'\wedge b)'\wedge b=0$.
Thus $(a\vee (a'\wedge b))+b=0$, which implies $a\vee (a'\wedge b)=b$ by \eqref{equ1}, proving that $\mathbf{L}$ is weakly orthomodular.

As for dual weak orthomodularity, if $a\ge b$, then $a'\le b'$, which yields
$(a\wedge (a'\vee b))'=a'\vee (a''\wedge b')=b'$ by \eqref{woml}.
Hence, since the complementation is injective by Lemma~\ref{lem3}, we have $a\wedge (a'\vee b)=b$, proving that $\mathbf{L}$ is dually weakly orthomodular.

As noted above, $\mathcal{M}\subseteq\mathcal{O}\cap\mathcal{O}^d$. Thus, if $\mathbf{L}\in\mathcal{M}_{DM}$, then the complementation is injective by Lemma~\ref{lem3}.
\end{proof}

\begin{lemma}\label{L:89}
Every $\mathbf{L}\in (\mathcal{O}\cap\mathcal{O}^d)_{DM}$ satisfies the quasi-identities 
\begin{gather}
\label{woml+}
x'\le y \;\Rightarrow\; x'\vee (x\wedge y)\approx y,\\
\label{dwoml+}
x'\ge y \;\Rightarrow\; x'\wedge (x\vee y)\approx y.
\end{gather}
In particular, every 
$\mathbf{L}\in (\mathcal{W}\cup\mathcal{M})_{DM}$ satisfies these quasi-identities.
\end{lemma}

\begin{proof}
Let $\mathbf{L}\in (\mathcal{O}\cap\mathcal{O}^d)_{DM}$ and suppose that $a'\le b$ for some $a,b\in L$. Then 
$a'=b\wedge (b'\vee a')=(a\wedge b)'\wedge b$ 
by dual weak orthomodularity \eqref{dwoml}. At the same time, 
$b=(a\wedge b)\vee ((a\wedge b)'\wedge b)$
by weak orthomodularity \eqref{woml}. Hence 
$b=(a\wedge b)\vee a'$, as desired.
Dually, $a'\ge b$ implies $b=(a\vee b)\wedge a'$.
The final statement follows from Lemma~\ref{lem3a}.
\end{proof}

Recall that, in a lattice $(L,\vee,\wedge)$, an element $a\in L$ is said to be
\begin{itemize}
\item
\emph{distributive} if 
$a\vee (x\wedge y)=(a\vee x)\wedge (a\vee y)$
for all $x,y\in L$,
\item
\emph{standard} if 
$x\wedge (a\vee y)=(x\wedge a)\vee (x\wedge y)$
for all $x,y\in L$, and
\item
\emph{neutral} if 
$(a\vee x)\wedge (a\vee y)\wedge (x\vee y)=
(a\wedge x)\vee (a\wedge y)\vee (x\wedge y)$
for all $x,y\in L$. 
\end{itemize}
\emph{Dually distributive} and \emph{dually standard} elements are defined dually, while being neutral is obviously self-dual.
See, e.g., \cite[Chapter III]{G11}.

In general, the relationships between the three (or five) types of elements are as follows:
neutral $\Rightarrow$ (dually) standard $\Rightarrow$ (dually) distributive.
In modular lattices, all these notions coincide, and in complemented modular lattices, neutral elements are those which have unique complements. 

In complemented lattices, neutral elements correspond to direct product decompositions.
Specifically, given a complemented lattice $\mathbf{L}=(L,\vee,\wedge,0,1)$, an element $e\in L$ is neutral if and only if $\mathbf{L}$ is isomorphic to the direct product of 
$\mathbf{[0,e]}=([0,e],\vee,\wedge,0,e)$ and
$\mathbf{[e,1]}=([e,1],\vee,\wedge,e,1)$;
the isomorphism in question is defined simply by $\eta\colon x\mapsto (x\wedge e,x\vee e)$.
We now extend this to lattices with (De Morgan) complementation.

\begin{lemma}\label{L:neutral=central}
Let $\mathbf{L}=(L,\vee,\wedge,{}',0,1)$ be a lattice with complementation. Let $e\in L$ be a neutral element and put
\[
x^\flat:=x'\wedge e \text{ for } x\in [0,e]
\quad\text{and}\quad
x^\sharp:=x'\vee e \text{ for } x\in [e,1].
\]
Then 
\begin{enumerate}[\rm(i)]
\item
$\mathbf{[0,e]}=([0,e],\vee,\wedge,{}^\flat,0,e)$ and
$\mathbf{[e,1]}=([e,1],\vee,\wedge,{}^\sharp,e,1)$ 
are lattices with complementation;
\item
if $\mathbf{L}$ satisfies De Morgan's laws, then $\mathbf{L}$ is isomorphic to $\mathbf{[0,e]}\times\mathbf{[e,1]}$;
\item
if $\mathbf{L}$ belongs to $\mathcal{W}$, then so do $\mathbf{[0,e]}$ and $\mathbf{[e,1]}$.
\end{enumerate}
\end{lemma}

\begin{proof}
(i) For any $x\in [0,e]$, we have 
$x\wedge x^\flat=x\wedge x'\wedge e=0$ and 
$x\vee x^\flat=x\vee (x'\wedge e)=e$, since $e$ is neutral (dually standard).
Analogously, for any $x\in [e,1]$,
$x\wedge x^\sharp=x\wedge (x'\vee e)=e$, since $e$ is neutral (standard) and
$x\vee x^\sharp=x\vee x'\vee e=1$.

(ii) Suppose that $\mathbf{L}$ satisfies De Morgan's laws.
Then the mapping $\eta\colon x\mapsto (x\wedge e,x\vee e)$ is a bounded lattice isomorphism which preserves complementation thanks to De Morgan's laws. Indeed, for any $x\in L$, we have
$(x\wedge e)^\flat=(x\wedge e)'\wedge e
=(x'\vee e')\wedge e=x'\wedge e$,
since $e$ is neutral (dually distributive), and
$(x\vee e)^\sharp=(x\vee e)'\vee e
=(x'\wedge e')\vee e=x'\vee e$,
since $e$ is neutral (distributive).
Hence
$(\eta(x))'=((x\wedge e)^\flat,(x\vee e)^\sharp)=
(x'\wedge e,x'\vee e)=\eta(x')$.

(iii) Suppose $\mathbf{L}\in\mathcal{W}$.
For all $x,y\in [0,e]$, we have $x^\flat\wedge y=x'\wedge e\wedge y=x'\wedge y$. Hence $(x^\flat\wedge y)\vee (x\wedge y^\flat)=0$ in $\mathbf{[0,e]}$ is the same as $x+y=0$ in $\mathbf{L}$, whence $x=y$. Thus $\mathbf{[0,e]}\in\mathcal{W}$.

Now, let $x,y\in [e,1]$. If $x^\sharp\wedge y=e$, then 
$x'\wedge y=x'\wedge x^\sharp\wedge y=x'\wedge e\le x'\wedge x=0$.
Hence $(x^\sharp\wedge y)\vee (x\wedge y^\sharp)=e$ in $\mathbf{[e,1]}$ implies $x+y=0$ in $\mathbf{L}$, which yields $x=y$. Thus $\mathbf{[e,1]}\in\mathcal{W}$.
\end{proof}

For example, for the lattice with complementation $\mathbf{H}_1\in\mathcal{M}_{DM}$ from Example~\ref{EXA2} we have
$\mathbf{H}_1\cong\mathbf{[0,e]}\times\mathbf{[e,1]}
\cong\mathbf{2}\times\mathbf{M}_3'$
and $\mathbf{H}_1\cong\mathbf{[0,d]}\times\mathbf{[d,1]}
\cong\mathbf{M}_3'\times\mathbf{2}$.
The elements $d$ and $e$ are neutral because they have unique complements, 
or in other words, because they determine direct product decompositions of the lattice;
see Figures~\ref{F3} and \ref{F6}.

\begin{lemma}\label{L:45}
Let $\mathbf{L}\in (\mathcal{O}\cup\mathcal{O}^d)_{DM}$. 
If an element $a\in L$ has a unique complement, then $a''=a$.
\end{lemma}

Note that the converse implication fails, e.g., in horizontal sums of (four-element) Boolean algebras.

\begin{proof}
Suppose $a\in L$ has a unique complement.
Let $b:=(a\vee a'')'\vee a''$ and note that $a\vee b=1$.
First, we have
$a\vee ((a\wedge b)'\wedge b)=a\vee b=1$
by weak orthomodularity \eqref{woml-}.
Since also $a\wedge (a\wedge b)'\wedge b=0$, the element 
$(a\wedge b)'\wedge b$ is a complement of $a$, and so
$a'=(a\wedge b)'\wedge b$.
Second, we have 
\[
(a\wedge b)'=
a'\vee \big((a\vee a'')'\vee a''\big)'=
a'\vee \big((a'\wedge a''')'\wedge a'''\big)=
a'\vee a'''=
(a\wedge a'')'
\]
again by \eqref{woml-}. Then
\[
a'=(a\wedge b)'\wedge b=
a'\vee \big((a\wedge b)'\wedge b\big)\ge
a'\vee \big((a\wedge b)'\wedge a''\big)=
(a\wedge b)'
\]
by weak orthomodularity \eqref{woml}, because $a'\le (a\wedge b)'$ owing to De Morgan's laws. This shows $a'=(a\wedge b)'=(a\wedge a'')'$, whence $a=a\wedge a''$. Now, by \eqref{woml}, $a\le a''$ implies $a=a\vee (a'\wedge a'')=a''$, as desired.
\end{proof}

\begin{lemma}\label{L:neutral}
Suppose that either 
{\rm(a)} $\mathbf{L}\in\mathcal{W}_{DM}$, or 
{\rm(b)} $\mathbf{L}\in\mathcal{M}_{DM}$ and satisfies identity \eqref{M31}.
Then the following are equivalent for every $a\in L$:
\begin{enumerate}[\rm(i)]
\item
the element $a$ is neutral;
\item
the element $a$ has a unique complement (in other words, $a'$ is the only complement of $a$);
\item
$a''=a$.
\end{enumerate}
\end{lemma}

We will see later that (b) is in fact a particular case of (a), because when $\mathbf{L}\in\mathcal{M}_{DM}$ satisfies \eqref{M31}, then $\mathbf{L}\in\mathsf{V}(\mathbf{M}_3')\subseteq\mathcal{W}_{DM}$ (see Theorem~\ref{T:M3axio} and/or Lemma~\ref{L:A1} in the appendix).

\begin{proof}
(a) Let $\mathbf{L}\in\mathcal{W}_{DM}$. 
By Lemma~\ref{lem3a} we know that $\mathbf{L}\in\mathcal{O}\cap\mathcal{O}^d$.
In complemented lattices, complements of neutral elements are also neutral, and neutral elements have unique complements.
Hence, if $a\in L$ is neutral, then $a'$ is neutral as well and therefore has a unique complement, while both $a''$ and $a$ are complements of $a'$, and so $a''=a$. 
Thus, (i) implies both (ii) and (iii), but also (ii) implies (iii) by Lemma~\ref{L:45}.

It remains to show that (iii) implies (i). For this direction, we first prove the following:

\begin{claim}
For every $e\in L$, if $e''=e$, then $e'$ is dually distributive and $e$ is distributive.
\end{claim}

We begin by showing that
\begin{equation}\label{E:Pa}
e\vee (e'\wedge x)=e\vee x
\end{equation}
for every $x\in L$. Put 
\[
y:=(e\vee x')\wedge x. 
\]
Note that
$e'\wedge y=e'\wedge x\wedge (e'\wedge x)'=0$.
Since $\mathbf{L}\in\mathcal{O}^d$, we have
$e\wedge x'=e\wedge ((e\vee x')'\vee x')=e\wedge y'$ 
by dual weak orthomodularity \eqref{dwoml-}.
Then
\[
e'\wedge \big((e\wedge x')\vee y\big)
=e'\wedge \big((e\wedge y')\vee y\big)
=e'\wedge \big((e'\vee y)'\vee y\big)=e'\wedge y=0
\]
again by \eqref{dwoml-}. On the other hand, we have
\[
e\wedge \big((e\wedge x')\vee y\big)'
=e\wedge (e\wedge y')'\wedge y'=0.
\] 
Thus $e+((e\wedge x')\vee y)=0$, which gives 
\[
e=(e\wedge x')\vee y
\]
by \eqref{equ1}. Since $\mathbf{L}\in\mathcal{O}$, we also have
\[
(e'\wedge x)\vee y=(e'\wedge x)\vee \big((e\vee x')\wedge x\big)
=(e'\wedge x)\vee \big((e'\wedge x)'\wedge x\big)=x
\]
by \eqref{woml}. Then, using $e=(e\wedge x')\vee y$ (twice) 
together with $(e'\wedge x)\vee y=x$ and $x\ge y$, we get
\[
e\vee (e'\wedge x)=(e\wedge x')\vee y\vee (e'\wedge x)
=(e\wedge x')\vee x=(e\wedge x')\vee y\vee x=e\vee x,
\]
proving \eqref{E:Pa}.

By \eqref{E:Pa} and De Morgan's laws, and since the complementation is injective, we easily obtain
\begin{equation}\label{E:Pb}
e'\wedge (e\vee x)=e'\wedge x
\end{equation}
for every $x\in L$.

Now, from \eqref{E:Pa} and \eqref{E:Pb} we deduce that the element $e'$ is dually distributive, whence the element $e$ is distributive owing to De Morgan's laws.
Indeed, for all $x,y\in L$, we have
\begin{align*}
e'\wedge (x\vee y)
&= e'\wedge (e\vee x\vee y)  
&& \text{by \eqref{E:Pb}}\\
&= e'\wedge (e\vee x\vee e\vee y)\\
&= e'\wedge \big(e\vee (e'\wedge x)\vee e\vee (e'\wedge y)\big) 
&& \text{by \eqref{E:Pa}}\\
&= e'\wedge \big((e'\wedge x)\vee (e'\wedge y)\big) 
&& \text{by \eqref{E:Pb}}\\
&= (e'\wedge x)\vee (e'\wedge y).
\end{align*}
This settles the claim.

We now return to the main proof. Suppose that $a\in L$ satisfies $a''=a$ and put $b:=a'$. Then $a$ is distributive, and since $b''=b$, the claim also ensures that $a=b'$ is dually distributive. 
Thus, $a$ is both distributive and dually distributive, and in order to conclude that $a$ is neutral, it suffices to show that $a$ satisfies the implication
\[
(a\vee x=a\vee y \text{ and } a\wedge x=a\wedge y) 
\;\Rightarrow\; x=y,
\]
for all $x,y\in L$.
Indeed, assuming $a\vee x=a\vee y$ and $a\wedge x=a\wedge y$, we have 
$a\vee (x\wedge y')=(a\vee x)\wedge (a\vee y')=(a\vee y)\wedge (a\vee y')=a\vee (y\wedge y')=a$, 
i.e., $a\ge x\wedge y'$, whence
$x\wedge y'=a\wedge x\wedge y'=a\wedge y\wedge y'=0$.
Similarly, $a\ge x'\wedge y$ and $x'\wedge y=0$.
Thus $x+y=0$, which implies $x=y$ by \eqref{equ1}.

(b) Suppose that $\mathbf{L}\in\mathcal{M}_{DM}$ satisfies identity \eqref{M31}. 
In complemented modular lattices, complements of neutral elements are neutral, and neutral elements are precisely those with unique complements. Thus, (i) and (ii) are equivalent to each other and imply (iii).

To complete the proof, we show that (iii) implies (ii).
Suppose that $b\in L$ is a complement of some $a\in L$ with $a''=a$. By De Morgan's laws and \eqref{M31}, and since $a\wedge b=0$, we have
\begin{align*}
a\wedge (a'\vee b)
&= a\wedge (a'\vee b)\wedge (a\wedge b)'\wedge (b\vee b')\\ 
&= a\wedge (a'\vee b)\wedge (a'\vee b')\wedge (b\vee b')\\ 
&\le (a\wedge a')\vee (a\wedge b)\vee (a\wedge b')
&& \text{by } \eqref{M31}\\
&= a\wedge b'\\
&= (a'\vee b)',
\end{align*}
whence $a\wedge (a'\vee b)=0$.
Then, by modularity, $a'=a'\vee (a\wedge (a'\vee b))=a'\vee b$, so that $b\le a'$.
As a consequence, again by modularity and since $a\vee b=1$, we obtain the desired conclusion:
$b=b\vee (a\wedge a')=(b\vee a)\wedge a'=a'$.
\end{proof}

\begin{corollary}
Suppose that either 
{\rm(a)} $\mathbf{L}\in\mathcal{W}_{DM}$, or 
{\rm(b)} $\mathbf{L}\in\mathcal{M}_{DM}$ and satisfies \eqref{M31}.
Then $\mathbf{L}$ is an ortholattice if and only if it is a Boolean algebra.
\end{corollary}

\begin{proof}
Let $\mathbf{L}$ be an ortholattice that satisfies (a) or (b). Then $a''=a$ for every $a\in L$. By Lemma~\ref{L:neutral}, this means that every element $a\in L$ is neutral, whence the lattice is distributive.
\end{proof}

\begin{corollary}
In every $\mathbf{L}\in\mathcal{W}_{DM}$, distributive, standard, and neutral elements coincide.
\end{corollary}

\begin{proof}
Let $\mathbf{L}\in\mathcal{W}_{DM}$.
If $a\in L$ is distributive, then 
$a=a\vee (a'\wedge a'')=a\vee a''$, i.e., $a\ge a''$, whence 
$a=a\wedge (a'\vee a'')=a''$ by dual weak orthomodularity \eqref{dwoml}. Thus, by Lemma~\ref{L:neutral}, the element $a$ is neutral.
\end{proof}

\begin{lemma}
Let $\mathbf{L}\in\mathcal{W}_{DM}$.
For any $a,b\in L$, we have $a+b=1$ if and only if $a$ is neutral and $a'=b$ (equivalently, $b$ is neutral and $b'=a$).
\end{lemma}

\begin{proof}
If $a\in L$ is neutral, then $a''=a$ by Lemma \ref{L:neutral}, and we have $a+a'=a'\vee (a\wedge a'')=a'\vee a=1$. 

Conversely, suppose that $a+b=1$ for some $a,b\in L$. By weak orthomodularity 
\eqref{woml} we have 
$(a\wedge b')\vee ((a\wedge b')'\wedge a)=a$, 
whence
\[
1=(a+b)\vee \big((a\wedge b')'\wedge a\big)
=(a'\wedge b)\vee (a\wedge b')\vee \big((a\wedge b')'\wedge a\big)
=(a'\wedge b)\vee a.
\]
Then
$a'=a'\wedge (a\vee (a'\wedge b))=a'\wedge b$ by \eqref{dwoml+},
i.e., $a'\le b$. Symmetrically, interchanging $a$ and $b$, we obtain $b'\le a$. Hence $1=a+b=a'\vee b'$, and so
$a=a\wedge (a'\vee b')=b'$ by dual weak orthomodularity \eqref{dwoml} as $a\ge b'$.
Symmetrically, again by \eqref{dwoml}, $b=b\wedge (b'\vee a')=a'$ as $b\ge a'$.
Finally, we have $a''=b'=a$ and $b''=a'=b$, which means that the elements $a$ and $b$ are neutral by Lemma~\ref{L:neutral}.
\end{proof}

\section{Congruence properties of $\mathcal{M}$}
\label{SEC5}

In this section, we describe basic congruence properties of the variety $\mathcal{M}$, and of $\mathcal{M}_{DM}$ in particular, that will be needed in the final section.
We first observe that in lattices with complementation that are both weakly orthomodular and dually weakly orthomodular, De Morgan's laws imply that their congruences coincide with the lattice congruences.

\begin{lemma}\label{lem=samecongr}
If $\mathbf{L}\in (\mathcal{O}\cap\mathcal{O}^d)_{DM}$, 
then $\mathrm{Con}\,\mathbf{L}=\mathrm{Con}\,(L,\vee,\wedge)$.
\end{lemma}

\begin{proof}
Let $\mathbf{L}\in (\mathcal{O}\cap\mathcal{O}^d)_{DM}$
and $\theta\in\mathrm{Con}\,(L,\vee,\wedge)$. 
For any $a,b\in L$, if $(a,b)\in\theta$, then  
$0=a\wedge b\wedge (a\wedge b)'\equiv_\theta a\wedge (a\wedge b)'$,
whence
$a' = a'\vee 0 \equiv_\theta a'\vee (a\wedge (a\wedge b)')=(a\wedge b)'$,
because $a'\le (a\wedge b)'$ and $\mathbf{L}$ satisfies \eqref{woml+} by Lemma~\ref{L:89}. Symmetrically, $(b',(a\wedge b)')\in\theta$.
Thus $(a',b')\in\theta$, proving that $\theta\in\mathrm{Con}\,\mathbf{L}$.
\end{proof}

De Morgan's laws \eqref{DM} are essential here; see Example~\ref{EXA2} where $\mathbf{H}_2$ (see also Example~\ref{EXA4}) is a lattice with complementation which belongs to $\mathcal{W}\cap\mathcal{M}$, but its congruence lattice is smaller than the congruence lattice of the lattice reduct.

Now, we want to determine which of the aforementioned varieties of lattices with complementation are congruence permutable. Since all these varieties are congruence distributive, being congruence permutable is equivalent to being arithmetical.

Recall that a ternary term $p(x,y,z)$ is called a \emph{Pixley term} for an algebra (or a class of similar algebras) if the algebra (or the class) satisfies the identities
\[
p(x,y,y)\approx p(x,y,x)\approx p(y,y,x)\approx x.
\]
It is well known that a variety is arithmetical if and only if it has a Pixley term. See, e.g., \cite[Chapter II]{BS} or \cite[Chapter 3]{CEL}.

\begin{lemma}\label{lem2}
The term
\[
p(x,y,z):=(x\wedge y\wedge z)\vee\big(x\wedge(x\wedge y)'\big)\vee\big(z\wedge(z\wedge y)'\big)
\]
is a Pixley term for a lattice with complementation $\mathbf{L}$ if and only if 
$\mathbf{L}\in\mathcal{O}$. Dually, the term
\[
q(x,y,z):=(x\vee y\vee z)\wedge\big(x\vee(x\vee y)'\big)\wedge\big(z\vee(z\vee y)'\big)
\]
is a Pixley term for $\mathbf{L}$ if and only if $\mathbf{L}\in\mathcal{O}^d$.
\end{lemma}

\begin{proof}
Let $\mathbf{L}$ be a lattice with complementation.
For any $a,b\in L$, we have
\begin{gather*}
p(a,b,a)=p(a,b,b)=p(b,b,a)
=(a\wedge b)\vee\big(a\wedge(a\wedge b)'\big)
\intertext{and, dually,}
q(a,b,a)=q(a,b,b)=q(b,b,a)
=(a\vee b)\wedge\big(a\vee(a\vee b)'\big).
\end{gather*}
Hence, the following conditions (a)--(c) are equivalent, and so are the dual conditions (a')--(c'):

\noindent
\begin{minipage}[t]{0.475\textwidth}
\begin{enumerate}[(a)]
\item
$p(x,y,z)$ is a Pixley term for $\mathbf{L}$;
\item
$(a\wedge b)\vee (a\wedge(a\wedge b)')=a$ \\
for all $a,b\in L$;
\item
$\mathbf{L}\in\mathcal{O}$.
\end{enumerate}
\end{minipage}
\hfill
\begin{minipage}[t]{0.475\textwidth}
\begin{enumerate}[(a')]
\item
$q(x,y,z)$ is a Pixley term for $\mathbf{L}$;
\item
$(a\vee b)\wedge (a\vee(a\vee b)')=a$ \\
for all $a,b\in L$;
\item
$\mathbf{L}\in\mathcal{O}^d$.
\end{enumerate}
\end{minipage}
\end{proof}

\begin{corollary}[cf.\ \cite{CL18}, Theorem 5.3]\label{C:arithm}
Every subvariety of $\mathcal{O}$, as well as every subvariety of $\mathcal{O}^d$, is arithmetical. In particular, the variety $\mathcal{M}$ is arithmetical.
\end{corollary}

Given the well-known properties of congruences of complemented modular lattices (namely, that congruences permute and correspond one-to-one to perspectivity-closed ideals, see e.g. \cite{G11}), it is no surprise that the variety $\mathcal{M}$ is arithmetical and regular at $0$.
In fact, $\mathcal{M}$ is congruence regular, as we now show.

Recall that a variety $\mathcal{V}$ is congruence regular if and only if, 
for some integer $n\ge 1$, there exist ternary terms $t_1(x,y,z),\dots,t_n(x,y,z)$ in the language of $\mathcal{V}$ such that $\mathcal{V}$ satisfies
\begin{equation}\label{E:R}
t_1(x,y,z)\approx\dots\approx t_n(x,y,z)\approx z
\;\Leftrightarrow\; x\approx y.
\end{equation}
See, e.g., \cite[Chapter 6]{CEL}.

\begin{lemma}\label{L:54}
Given a binary term $x\oplus y$, a lattice with complementation satisfies condition \eqref{E:R} for $n=2$ and the ternary terms
\[
p_1(x,y,z):=(x\oplus y)\vee z
\quad\text{and}\quad
p_2(x,y,z):=(x\oplus y)'\wedge z
\]
if and only if it satisfies 
\begin{equation}\label{E:D}
x\oplus y\approx 0\;\Leftrightarrow\; x\approx y.
\end{equation}
Dually, given a binary term $x\odot y$, a lattice with complementation satisfies condition \eqref{E:R} for $n=2$ and the ternary terms
\[
q_1(x,y,z):=(x\odot y)\wedge z
\quad\text{and}\quad
q_2(x,y,z):=(x\odot y)'\vee z
\]
if and only if it satisfies 
\begin{equation}\label{E:DD}
x\odot y\approx 1\;\Leftrightarrow\; x\approx y.
\end{equation}
\end{lemma}

\begin{proof}
Let $\mathbf{L}$ be a lattice with complementation.
Suppose that $\mathbf{L}$ satisfies condition \eqref{E:R} for $p_1(x,y,z)$ and $p_2(x,y,z)$, i.e., it satisfies
\begin{equation}\label{E:RP12}
p_1(x,y,z)\approx p_2(x,y,z)\approx z
\;\Leftrightarrow\; x\approx y.
\end{equation}
Let $a,b\in L$. If $a\oplus b=0$, then 
$p_1(a,b,0)=p_2(a,b,0)=0$, whence $a=b$ by \eqref{E:RP12}.
Clearly, $a\oplus a=p_1(a,a,0)=0$ again by \eqref{E:RP12}. 
Thus $\mathbf{L}$ satisfies \eqref{E:D}.

Conversely, suppose that $\mathbf{L}$ satisfies condition \eqref{E:D}. Let $a,b,c\in L$.
If $p_1(a,b,c)=p_2(a,b,c)=c$, i.e., $(a\oplus b)\vee c=(a\oplus b)'\wedge c=c$, then we have $a\oplus b\le c\le (a\oplus b)'$, which implies $a\oplus b=0$, and so $a=b$ by \eqref{E:D}.
On the other hand, by \eqref{E:D} we have $a\oplus a=0$, and hence $p_1(a,a,c)=p_2(a,a,c)=c$.
Thus $\mathbf{L}$ satisfies \eqref{E:RP12}.

The arguments for the second statement are dual.
\end{proof}

\begin{corollary}[cf.\ \cite{CL18}, Theorem 5.3]\label{C:reg}
Every subvariety of $\mathcal{O}$, as well as every subvariety of $\mathcal{O}^d$, is congruence regular. In particular, the variety $\mathcal{M}$ is congruence regular.
\end{corollary}

\begin{proof}
For a subvariety of $\mathcal{O}$ we may take the term
\[
x\oplus y:=(x\vee y)\wedge (x\wedge y)',
\]
and for a subvariety of $\mathcal{O}^d$, the dual term
\[
x\odot y:=(x\wedge y)\vee (x\vee y)'.
\]
Indeed, let $\mathbf{L}\in\mathcal{O}$ and $a,b\in L$.
Clearly, $a\oplus a=a\wedge a'=0$. If $a\oplus b=0$, then 
$a\wedge b=(a\wedge b)\vee (a\oplus b)
=(a\wedge b)\vee ((a\wedge b)'\wedge (a\vee b))=a\vee b$, 
and so $a=b$.
Thus condition \eqref{E:D} is satisfied.
The argument for $\mathbf{L}\in\mathcal{O}^d$ with the term
$x\odot y$ is dual.
\end{proof}

\section{The varieties generated by $\mathbf{M}_n'$}
\label{SEC6}

The first goal of this section is to axiomatize the variety of lattices with complementation generated by $\mathbf{M}_3'$.
It is no surprise that, to some extent, we rely on J\'onsson's axiomatization of the varieties of lattices generated by the $\mathbf{M}_n$'s (see \cite{Jo} or \cite[Chapter 5]{JR}):

For every integer $n\ge 3$, the variety generated by the lattice $\mathbf{M}_n$ is axiomatized, relative to modular lattices, by the identities 
\begin{gather}
\label{jo}
u\wedge \big(x\vee (y\wedge z)\big)\wedge (y\vee z)\le
x\vee (u\wedge y)\vee (u\wedge z)
\intertext{and}
\tag{$\mathrm{M}_n$}\label{mn}
x_0\wedge\bigwedge_{1\le i<j\le n} (x_i\vee x_j) \le 
\bigvee_{1\le i\le n} (x_0\wedge x_i).
\end{gather}
In particular, \eqref{jo} is superfluous for $n=3$, whence the variety generated by $\mathbf{M}_3$ is axiomatized, relative to modular lattices, by 
\[
\tag{$\mathrm{M}_3$}\label{M3}
u\wedge (x\vee y)\wedge (x\vee z)\wedge (y\vee z)
\le (u\wedge x)\vee (u\wedge y)\vee (u\wedge z).
\]

This identity does not, however, suffice to axiomatize the variety generated by $\mathbf{M}_3'$ relative to modular lattices with complementation. For example, $\mathbf{H}_2$ from Example~\ref{EXA2} (also see Example~\ref{EXA4}) is a subdirectly irreducible modular lattice with complementation that satisfies \eqref{M3} but obviously does not belong to the variety $\mathsf{V}(\mathbf{M}_3')$. 

The situation changes when we consider modular lattices with De~Morgan complementation, since in this case, based on Lemma~\ref{lem=samecongr}, we obtain the following:

\begin{theorem}\label{T:M3axio}
Relative to the variety $\mathcal{M}_{DM}$ of modular lattices with \linebreak De~Morgan complementation, the variety $\mathsf{V}(\mathbf{M}_3')$ is axiomatized by the identity \eqref{M3}.
\end{theorem}

\begin{proof}
Let $\mathcal{K}$ be the subvariety of $\mathcal{M}_{DM}$ defined by \eqref{M3}. Let $\mathbf{L}\in\mathcal{K}$ be subdirectly irreducible. Since $\mathcal{M}\subseteq\mathcal{O}\cap\mathcal{O}^d$, $\mathbf{L}$ and its lattice reduct $(L,\vee,\wedge)$ have the same congruences, i.e., 
$\mathrm{Con}\,\mathbf{L}=\mathrm{Con}\,(L,\vee,\wedge)$, by Lemma~\ref{lem=samecongr}. Then $(L,\vee,\wedge)$ is a subdirectly irreducible modular lattice satisfying \eqref{M3}, i.e., a subdirectly irreducible member of the lattice variety $\mathsf{V}(\mathbf{M}_3)$.
Consequently, $(L,\vee,\wedge)$ is either the two-element chain or $\mathbf{M}_3$.
In the former case, $\mathbf{L}$ is the Boolean algebra $\mathbf{2}$, while in the latter case, $\mathbf{L}$ is $\mathbf{M}_3'$ because on $\mathbf{M}_3$ there is (up to isomorphism) a unique De Morgan complementation.
Hence $\mathcal{K}=\mathsf{V}(\mathbf{M}_3')$.
\end{proof}

In order to axiomatize the variety generated by $\mathbf{M}_n'$ for $n\ge 4$ (cf.\ Example~\ref{EXA1}), we again invoke the aforementioned J\'onsson's axiomatization, although we cannot simply repeat the arguments used in the proof of Theorem~\ref{T:M3axio}.
The reason is that the identities \eqref{jo} and \eqref{mn} do not axiomatize $\mathsf{V}(\mathbf{M}_n')$ relative to $\mathcal{M}_{DM}$, since \eqref{mn} neither excludes lattices $\mathbf{M}_p$ with $p<n$ nor imposes any additional restrictions on complementation. 

In other words, assuming $\mathbf{L}\in\mathcal{M}_{DM}$ is subdirectly irreducible and satisfies \eqref{jo} and \eqref{mn} for some $n\ge 4$, we can only conclude that the lattice reduct of $\mathbf{L}$ is either the two-element chain or $\mathbf{M}_p$ for some $p\in\{3,\dots,n\}$, but we do not know enough about the complementation.

Therefore, we need another identity to single out $\mathbf{M}_n'$. To this end, for every integer $n\ge 3$, we consider the term
\[
\tau_n(x):=\bigwedge_{2\le k<n} \big(x\vee x^{(k)}\big),
\]
where the derivative-like notation $x^{(k)}$ is shorthand for $x^{\prime\prime\cdots\prime}$ with $k$ occurrences of $'$. More formally, 
\[
x^{(1)}:=x' \quad\text{and}\quad x^{(k)}:=\big(x^{(k-1)}\big)' \text{ for } k\ge 2.
\]
For example, 
\[
\tau_3(x)=x\vee x'',\quad \tau_4(x)=(x\vee x'')\wedge (x\vee x'''),\quad\text{etc.}
\]
Clearly, $\tau_n(0)=0$ and $\tau_n(1)=1$ hold in all lattices with complementation.

It is worth noting that in $\mathbf{M}_p'$ with $p\ge n$ we have
\begin{equation}\label{tau}
\tau_n(a)=
\begin{cases}
0 & \text{for } a=0,\\
1 & \text{otherwise.}
\end{cases}
\end{equation}
Indeed, $\tau_n(a_i)=1$ because, for all $k\in\{2,\dots,n-1\}$, we have $a_i^{(k)}\neq a_i$ and so $a_i\vee a_i^{(k)}=1$.
On the other hand, \eqref{tau} fails in $\mathbf{M}_p'$ with $p<n$, since $a_i^{(p)}=a_i$ and so $\tau_n(a_i)=a_i$.
Hence, $\mathbf{M}_p'$ with $p<n$ actually satisfies the identity $\tau_n(x)\approx x$.

Since $0$ and $1$ are the only neutral elements of the lattices $\mathbf{M}_p$ with $p\ge 3$, we may summarize the above observations as follows: 

\begin{lemma}
For any pair of integers $n,p\ge 3$, the following conditions are equivalent:
\begin{enumerate}[\rm(i)]
\item
$n\le p$;
\item
all elements of the form $\tau_n(a)$ in $\mathbf{M}_p'$ are neutral; 
\item
$\mathbf{M}_p'$ satisfies the identity
\[\tag{$\mathrm{T}_n$}\label{tn}
x\wedge (\tau_n(y)\vee z)\approx
(x\wedge\tau_n(y))\vee (x\wedge z).
\]
\end{enumerate}
\end{lemma}

Strictly speaking, satisfaction of \eqref{tn} means that the elements $\tau_n(a)$ are standard, but in any modular lattice, standard and neutral elements coincide.

Next, we note that if the element $\tau_n(a)$ is neutral, then it equals $a\vee a''$. More precisely:

\begin{lemma}
Let $\mathbf{L}\in\mathcal{M}$. For any integer $n\ge 3$ and any element $a\in L$, if the element $\tau_n(a)$ is neutral, then
\begin{equation}
\label{aak}
a\vee a^{(k)}=
\begin{cases}
\tau_n(a) & \text{for even } k\in\{2,\dots,n-1\},\\
1 & \text{for odd } k\in\{3,\dots,n-1\} \text{ if } n\ge 4. 
\end{cases}
\end{equation}
\end{lemma}

\begin{proof}
Let $\tau_n(a)$ be a neutral element.
There is nothing to prove for $n=3$, so let $n\ge 4$. 
We have $\tau_n(a)\vee a'\ge a\vee a'=1$, and hence
\begin{align*}
a'' &= a''\wedge (\tau_n(a)\vee a')\\
&=(a''\wedge\tau_n(a))\vee (a''\wedge a')
&& \text{as } \tau_n(a) \text{ is neutral (standard)}\\
&=a''\wedge\tau_n(a),
\end{align*}
i.e., $a''\le\tau_n(a)$ and $a\vee a''\le\tau_n(a)$.
Obviously, $a\vee a''\ge\tau_n(a)$, and so $a\vee a''=\tau_n(a)$.
Now, $a\vee a'''\ge\tau_n(a)=a\vee a''$ implies $a\vee a'''=1$.
This settles \eqref{aak} for $n=4$.

Proceeding inductively, we prove that \eqref{aak} holds for $n\ge 5$.
First, assume that $a\vee a^{(k)}=\tau_n(a)$ for some even $k\in\{2,\dots,n-2\}$.
Then $a\vee a^{(k+1)}\ge\tau_n(a)=a\vee a^{(k)}$ yields $a\vee a^{(k+1)}=1$.
If $k+2<n$, then $\tau_n(a)\vee a^{(k+1)}=a\vee a^{(k)}\vee a^{(k+1)}=1$, and hence
\begin{align*}
a^{(k+2)} &= a^{(k+2)}\wedge\big(\tau_n(a)\vee a^{(k+1)}\big)\\
&=\big(a^{(k+2)}\wedge\tau_n(a)\big)\vee\big(a^{(k+2)}\wedge a^{(k+1)}\big)
&& \text{as } \tau_n(a) \text{ is neutral (standard)}\\
&=a^{(k+2)}\wedge\tau_n(a),
\end{align*}
i.e., $a^{(k+2)}\le\tau_n(a)$ and $a\vee a^{(k+2)}\le\tau_n(a)$. Obviously, $a\vee a^{(k+2)}\ge\tau_n(a)$, and so $a\vee a^{(k+2)}=\tau_n(a)$.

Second, assume that $a\vee a^{(k)}=1$ for some odd $k\in\{3,\dots,n-2\}$. Then
$\tau_n(a)\vee a^{(k)}\ge a\vee a^{(k)}=1$, and repeating the above calculations with $a^{(k+1)}$ and $a^{(k)}$ in place of $a^{(k+2)}$ and $a^{(k+1)}$, respectively, we obtain $a^{(k+1)}\le\tau_n(a)$, whence $a\vee a^{(k+1)}=\tau_n(a)$, as required.
\end{proof}

Now, we can axiomatize the varieties generated by $\mathbf{M}_n'$ for any $n\ge 3$; the case $n=3$ is included, though it is already covered by Theorem~\ref{T:M3axio} (see also the appendix).

\begin{theorem}\label{T:Mnaxio}
For every integer $n\ge 3$, the variety $\mathsf{V}(\mathbf{M}_n')$ is axiomatized, relative to the variety $\mathcal{M}_{DM}$, by the identities \eqref{mn} and \eqref{tn}.
\end{theorem}

\begin{proof}
Given $n\ge 3$, let $\mathcal{K}$ be the subvariety of $\mathcal{M}_{DM}$ defined by identities \eqref{mn} and \eqref{tn}.
The algebra $\mathbf{M}_n'$ clearly satisfies \eqref{mn}, and it satisfies \eqref{tn} by the preceding lemma. Hence $\mathsf{V}(\mathbf{M}_n')\subseteq\mathcal{K}$.

For the converse, suppose that $\mathbf{L}$ is a subdirectly irreducible member of  $\mathcal{K}$ other than the Boolean algebra $\mathbf{2}$. We aim at showing that $\mathbf{L}$ is isomorphic to $\mathbf{M}_n'$.

In view of Lemma~\ref{L:neutral=central}~(ii), neutral elements correspond to direct product decompositions, and hence the only neutral elements in $\mathbf{L}$ are $0$ and $1$.
The satisfaction of \eqref{tn} guarantees that the element $\tau_n(a)$ is neutral for every $a\in L\setminus\{0,1\}$.
Since $\tau_n(a)\ge a$, we have $\tau_n(a)=1$, whence 
\begin{equation}\label{bsfhjs}
a\vee a^{(k)}=1
\end{equation}
for all $k\in\{2,\dots,n-1\}$. This holds for every $a\in L\setminus\{0,1\}$. Hence, since $a\in L\setminus\{0,1\}$ if and only if $a'\in L\setminus\{0,1\}$, by replacing $a$ with $a'$ in \eqref{bsfhjs}, we obtain
$1=a'\vee (a')^{(k)}=a'\vee a^{(k+1)}=(a\wedge a^{(k)})'$, which yields
\[
a\wedge a^{(k)}=0
\]
for all $k\in\{2,\dots,n-1\}$. Thus, for every $a\in L\setminus\{0,1\}$, we actually have
\begin{equation}\label{nobbvzvw}
a\vee a^{(k)}=1 \quad\text{and}\quad a\wedge a^{(k)}=0 
\end{equation}
for all $k\in\{1,\dots,n-1\}$.

Consequently, the elements $a,a',a'',\dots,a^{(n-1)}$ are pairwise distinct, since 
$a$ is clearly distinct from its complements $a',a'',\dots,a^{(n-1)}$, which themselves are pairwise distinct owing to the injectivity of complementation (cf.\ Lemma~\ref{lem3a}).

Now, suppose by way of contradiction that  
$(L,\vee,\wedge)$, the lattice reduct of $\mathbf{L}$, is of length $>2$. Then there exist $a,b\in L$ with $0<a<b<1$. Since complementation is antitone and injective, we have $0<a''<b''<1$ and the elements $0,a,b,a'',b''$ and $1$ form a sublattice isomorphic to the benzene ring (see Figure~\ref{fig-benzene}), which contradicts modularity.
Indeed, by \eqref{nobbvzvw}, 
$a\leq b''$ would yield $a\leq b\wedge b''=0$,
$a\geq b''$ would yield $a=a\vee a''=1$,
$b\leq a''$ would yield $b=b\wedge b''=0$, and 
$b\geq a''$ would yield $b\geq a\vee a''=1$.

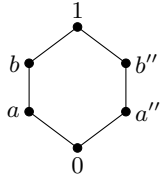
\begin{figure}[ht]
\centering
\begin{tikzpicture}[scale=0.8,radius=2pt,font=\small]
\filldraw (1,0) circle node [below] {$0$} coordinate (0);
\filldraw (0.2,0.6) circle node [left] {$a$} coordinate (a);
\filldraw (0.2,1.4) circle node [left] {$b$} coordinate (b);
\filldraw (1.8,0.6) circle node [right] {$a''$} coordinate (a'');
\filldraw (1.8,1.4) circle node [right] {$b''$} coordinate (b'');
\filldraw (1,2) circle node [above] {$1$} coordinate (1);
\draw (0) -- (a) -- (b) -- (1) -- (b'') -- (a'') -- (0);
\end{tikzpicture}
\caption{The benzene ring}\label{fig-benzene}
\end{figure}

Therefore, $(L,\vee,\wedge)$ is a modular lattice of length $2$ satisfying \eqref{mn}, and so it is isomorphic to the lattice $\mathbf{M}_p$ for some $p\le n$. 
Recalling that, for every $a\in L\setminus\{0,1\}$, the elements $a,a',\dots,a^{(n-1)}$ are pairwise distinct, we conclude that $p=n$, i.e., $(L,\vee,\wedge)$ is isomorphic to $\mathbf{M}_n$.
Finally, we have $a^{(n)}=a$ because, again owing to the injectivity of complementation, $a^{(n)}$ cannot be any of the elements $a',\dots,a^{(n-1)}$. Thus, the lattice with complementation $\mathbf{L}$ is isomorphic to $\mathbf{M}_n'$, as desired.
\end{proof}

In particular, for $n=3$ we have that $\mathsf{V}(\mathbf{M}_3')$ is axiomatized by \eqref{M3} and
\[
\tag{$\mathrm{T}_3$}\label{T3}
x\wedge (y\vee y''\vee z)\approx
\big(x\wedge (y\vee y'')\big)\vee (x\wedge z).
\]
The latter identity becomes redundant in light of Theorem~\ref{T:M3axio}, but it is not straightforward to derive it from \eqref{M3} and the axioms of $\mathcal{M}_{DM}$ (see the appendix).
For $n\ge 4$, as we have already mentioned, \eqref{tn} cannot be dropped. For example, for $n=4$, the identity \thetag{$\mathrm{T}_4$} is satisfied not only by $\mathbf{M}_4'$ but also by two other simple algebras that do not belong to $\mathsf{V}(\mathbf{M}_4')$, namely $\mathbf{M}_3'$ and the horizontal sum of two copies of the Boolean algebra $\mathbf{M}_2'\cong\mathbf{2}^2$.

Recall that a \emph{discriminator term} for an algebra $\mathbf{A}$ is a ternary term $t(x,y,z)$ in the language of $\mathbf{A}$ such that $\mathbf{A}$ satisfies
\[
t(a,b,c)=
\begin{cases}
a & \text{if } a\ne b,\\
c & \text{if } a=b.
\end{cases}
\]
A finite algebra having a discriminator term is called a \emph{quasiprimal algebra}. 
It is known that a finite algebra $\mathbf{A}$ is quasiprimal if and only if 
(i)~$\mathsf{V}(\mathbf{A})$ is an arithmetical variety and 
(ii)~all subalgebras of $\mathbf{A}$ are simple. 
See, e.g., \cite[Chapter IV]{BS}.

In the following lemma, for a given $n\ge 3$, $x\oplus y$ is any ``difference'' for $\mathbf{M}_n'$, i.e., a binary term such that 
$x\oplus y\approx 0$ $\Leftrightarrow$ $x\approx y$ 
holds in $\mathbf{M}_n'$ (cf.\ condition \eqref{E:D} in Lemma \ref{L:54}). 
For instance, we may take 
\[
x\oplus y:=(x\vee y)\wedge (x\wedge y)'.
\] 
For $n=3$ (and only in this case), we may take $x\oplus y:=x+y$.

\begin{lemma}
For every integer $n\ge 3$, $\mathbf{M}_n'$ is a quasiprimal algebra.
The term
\[
t(x,y,z)=
\big(\tau_n(x\oplus y)\wedge x\big)\vee\big((\tau_n(x\oplus y))'\wedge z\big)
\]
is a discriminator term for $\mathbf{M}_n'$.
\end{lemma}

\begin{proof}
In $\mathbf{M}_n'$, $\tau_n(a\oplus a)=0$ as $a\oplus a=0$, and $\tau_n(a\oplus b)=1$ for $a\ne b$ as $a\oplus b\ne 0$. Hence, we have
$t(a,a,c)=(0\wedge a)\vee (0'\wedge c)=c$, while 
$t(a,b,c)=(1\wedge a)\vee (1'\wedge c)=a$
for $a\ne b$.
\end{proof}

\begin{lemma}\label{lem1}
For every integer $n\ge3$, the automorphism group of $\mathbf{M}_n'$ is isomorphic to $(\mathbb{Z}_n,+)$.
\end{lemma}

\begin{proof}
Every automorphism $\alpha$ of $\mathbf{M}_n'$ is determined by the image of $a_0$ (or, actually, of any fixed $a_k$). Indeed, if $\alpha(a_0)=a_i$, then $\alpha(a_1)=\alpha(a_0')=a_i'=a_{i+1}$ etc. Thus, we have $\alpha(a_k)=a_{i+k}$ for every $k\in\{0,\dots,n-1\}$.
Now, if $\alpha_i$ and $\alpha_j$ are the automorphisms taking $a_0$ to $a_i$ and $a_j$, respectively, then $\alpha_i\circ\alpha_j$ is the automorphism taking $a_0$ to $a_{i+j}$, i.e., the automorphism $\alpha_{i+j}$.
\end{proof}

Following \cite{Be}, we can easily describe the free $k$-generator algebras in the variety $\mathsf{V}(\mathbf{M}_n')$ for any $k\ge 1$ and $n\ge 3$:

\begin{theorem}\label{T:FREE}
For any given integers $k\ge 1$ and $n\ge 3$, the free algebra $\mathbf{F}_{\mathsf{V}(\mathbf{M}_n')}(k)$ is isomorphic to
\[
(\mathbf{M}_n')^q\times\mathbf{2}^{2^k}
\]
with $q=((n+2)^k-2^k)/n$. In particular, $\mathbf{F}_{\mathsf{V}(\mathbf{M}_n')}(1)$ is isomorphic to
\[
\mathbf{M}_n'\times\mathbf{2}^2.
\]
\end{theorem}

Note that the statement fails for $n=2$, because $\mathsf{V}(\mathbf{M}_2')$ is the variety of Boolean algebras, so the free $k$-generator algebra is just $\mathbf{2}^{2^k}$.

\begin{proof}
The aforementioned properties of $\mathbf{M}_n'$ (it is a quasiprimal algebra) guarantee that
\[
\mathbf{F}_{\mathsf{V}(\mathbf{M}_n')}(k)\cong
(\mathbf{M}_n')^q\times\mathbf{2}^r
\] 
for certain integers $q,r\ge 1$ (cf.\ \cite{Be}).
The exponent in question is the number of valuations of the free generators to the respective algebra, divided by the number of automorphisms of that algebra. Here, a valuation means a mapping $v$ sending the free generators $x_1,\dots,x_k$ to the given algebra, with the property that the images $v(x_1),\dots,v(x_k)$ generate that algebra.
For $\mathbf{M}_n'$ we have $q=((n+2)^k-2^k)/n$ because the algebra 
\begin{itemize}
\item
is generated by any of the elements $a_0,\dots,a_{n-1}$, so that the images $v(x_1),\dots,v(x_k)$ generate $\mathbf{M}_n'$ unless $v(x_1),\dots,v(x_k)\in\{0,1\}$, and hence there are $(n+2)^k-2^k$ valuations of $x_1,\dots,x_k$ to $\mathbf{M}_n'$,
and 
\item
has $n$ automorphisms by Lemma~\ref{lem1}.
\end{itemize}
For $\mathbf{2}$ we obviously have $r=2^k$.
\end{proof}

For example, the free $1$- and $2$-generator algebras in the varieties $\mathsf{V}(\mathbf{M}_3')$ and $\mathsf{V}(\mathbf{M}_4')$ are as follows:
\begin{align*}
&\mathbf{F}_{\mathsf{V}(\mathbf{M}_3')}(1) \cong
\mathbf{M}_3'\times\mathbf{2}^2\text{ (see Figure~\ref{fig-FREE1})}, 
&&\mathbf{F}_{\mathsf{V}(\mathbf{M}_4')}(1) \cong
\mathbf{M}_4'\times\mathbf{2}^2, \\
&\mathbf{F}_{\mathsf{V}(\mathbf{M}_3')}(2) \cong 
(\mathbf{M}_3')^7\times\mathbf{2}^4, 
&&\mathbf{F}_{\mathsf{V}(\mathbf{M}_4')}(2) \cong 
(\mathbf{M}_4')^8\times\mathbf{2}^4.
\end{align*}

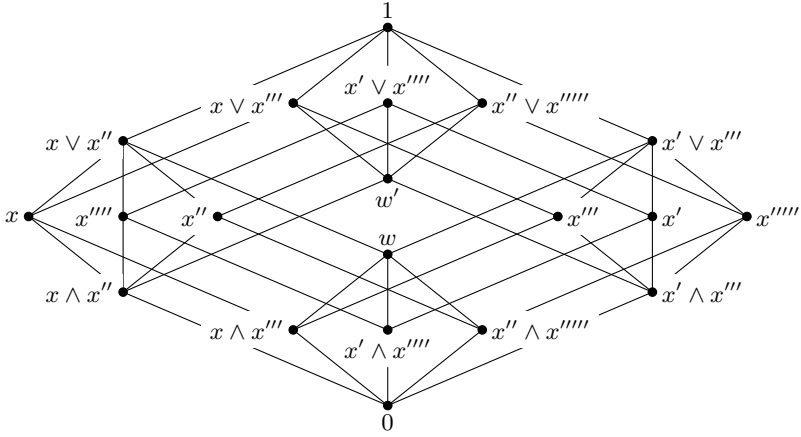
\begin{figure}[ht]
\centering
\begin{tikzpicture}[scale=1,radius=1.6pt,font=\small]
\coordinate (0) at (0,0);
\coordinate (1) at (0,5);
\coordinate (x) at (-4.75,2.5,0);
\coordinate (x'''') at (-3.5,2.5);
\coordinate (x'') at (-2.25,2.5);
\coordinate (x') at (3.5,2.5);
\coordinate (x''') at (2.25,2.5);
\coordinate (x''''') at (4.75,2.5);
\coordinate (x v x'') at (-3.5,3.5);
\coordinate (x ^ x'') at (-3.5,1.5);
\coordinate (x' v x''') at (3.5,3.5);
\coordinate (x' ^ x''') at (3.5,1.5);
\coordinate (x v x''') at (-1.25,4);
\coordinate (x' v x'''') at (0,4);
\coordinate (x'' v x''''') at (1.25,4);
\coordinate (w') at (0,3);
\coordinate (x ^ x''') at (-1.25,1);
\coordinate (x' ^ x'''') at (0,1);
\coordinate (x'' ^ x''''') at (1.25,1);
\coordinate (w) at (0,2);
\draw (x) -- (x ^ x'') -- (x'''') -- (x v x'') -- (x);
\draw (x ^ x'') -- (x'') -- (x v x'');
\draw (x' ^ x''') -- (x''''') -- (x' v x''') -- (x''') -- (x' ^ x''') -- (x') -- (x' v x''');
\draw (x' ^ x''') -- (0) -- (x ^ x'');
\draw (x' v x''') -- (1) -- (x v x'');
\draw (x) -- (x v x''') -- (x''');
\draw (x'''') -- (x' v x'''') -- (x');
\draw (x'') -- (x'' v x''''') -- (x''''');
\draw (w') -- (x'' v x''''') -- (1) -- (x v x''') -- (w') -- (x' v x'''') -- (1);
\draw (x) -- (x ^ x''') -- (x''');
\draw (x'''') -- (x' ^ x'''') -- (x');
\draw (x'') -- (x'' ^ x''''') -- (x''''');
\draw (w) -- (x'' ^ x''''') -- (0) -- (x ^ x''') -- (w) -- (x' ^ x'''') -- (0);
\draw (x v x'') -- (w) -- (x' v x''');
\draw (x ^ x'') -- (w') -- (x' ^ x''');
\draw (0) node [below] {$0$};
\draw (1) node [above] {$1$};
\draw (x) node [left] {$x$};
\draw (x'''') node [left] {$x''''$};
\draw (x'') node [fill=white,left] {$x''$};
\draw (x') node [right] {$x'$};
\draw (x''') node [fill=white,right] {$x'''$};
\draw (x''''') node [right] {$x'''''$};
\draw (x v x'') node [fill=white,left] {$x\vee x''$};
\draw (x ^ x'') node [fill=white,left] {$x\wedge x''$};
\draw (x' v x''') node [fill=white,right] {$x'\vee x'''$};
\draw (x' ^ x''') node [fill=white,right] {$x'\wedge x'''$};
\draw (x v x''') node [fill=white,left] {$x\vee x'''$};
\draw (x' v x'''') node [fill=white,above,rounded corners] {$x'\vee x''''$};
\draw (x'' v x''''') node [fill=white,right] {$x''\vee x'''''$};
\draw (w') node [below] {$w'$};
\draw (x ^ x''') node [fill=white,left] {$x\wedge x'''$};
\draw (x' ^ x'''') node [fill=white,below,rounded corners] {$x'\wedge x''''$};
\draw (x'' ^ x''''') node [fill=white,right] {$x''\wedge x'''''$};
\draw (w) node [above] {$w$};
\filldraw (0) circle;
\filldraw (1) circle;
\filldraw (x)  circle;
\filldraw (x'''') circle;
\filldraw (x'') circle;
\filldraw (x') circle;
\filldraw (x''') circle;
\filldraw (x''''') circle;
\filldraw (x v x'') circle;
\filldraw (x ^ x'') circle;
\filldraw (x' v x''') circle;
\filldraw (x' ^ x''') circle;
\filldraw (x v x''') circle;
\filldraw (x' v x'''') circle;
\filldraw (x'' v x''''') circle;
\filldraw (w') circle;
\filldraw (x ^ x''') circle;
\filldraw (x' ^ x'''') circle;
\filldraw (x'' ^ x''''') circle;
\filldraw (w) circle;
\end{tikzpicture}
\caption{The free algebra $\mathbf{F}_{\mathsf{V}(\mathbf{M}_3')}(1)$,
with $w={(x\vee x'')\wedge (x'\vee x''')}$}
\label{fig-FREE1}
\end{figure}

Finally, the algebras $\mathbf{M}_n'$ can be used to show that the variety $\mathcal{M}_{DM}$ has as many subvarieties as possible.

\begin{theorem}\label{T:unc}
There are uncountably many subvarieties of $\mathcal{M}_{DM}$.
\end{theorem}

\begin{proof}
For every $\emptyset\ne N\subseteq\{3,4,5,\dots\}$, let
$\mathcal{V}_N=\mathsf{V}(\mathbf{M}_n'\mid n\in N)$.
It suffices to show that the mapping 
\[
N\mapsto\mathcal{V}_N
\]
is injective.
Suppose that $\mathbf{M}_p'\in\mathcal{V}_N$ but $p\notin N$, for some integer $p\ge 3$. Then 
$\mathbf{M}_p'\in\mathsf{HSP_U}(\mathbf{M}_n'\mid n\in N)$, i.e., 
$\mathbf{M}_p'\in\mathsf{H}(\mathbf{K})$ for some 
$\mathbf{K}\in\mathsf{S}(\mathbf{L})$ where 
$\mathbf{L}\in\mathsf{P_U}(\mathbf{M}_n'\mid n\in N)$. 
Neither $\mathbf{K}$ nor $\mathbf{L}$ is a Boolean algebra. 

Since the lattices $\mathbf{M}_n$ are of length $2$ and width $\ge 3$ (both first-order properties), the same holds for the lattice reduct of the ultraproduct $\mathbf{L}$, which is therefore (isomorphic to) $\mathbf{M}_\lambda$ for some $\lambda\ge 3$. Similarly, since $\mathbf{K}$ is not a Boolean algebra, its lattice reduct is (isomorphic to) $\mathbf{M}_\kappa$ for some $\kappa\ge 3$.
Then $\mathbf{K}$ is a simple algebra, so $\mathbf{M}_p'\in\mathsf{H}(\mathbf{K})$ must be isomorphic to $\mathbf{K}\in\mathsf{S}(\mathbf{L})$. 
However, since $p\notin N$, no $\mathbf{M}_n'$ has a subalgebra isomorphic to $\mathbf{M}_p'$, and hence neither does $\mathbf{L}$ (not containing a fixed finite subalgebra is another first-order property). This is a contradiction.
\end{proof}

\section*{Appendix}

As promised in the remarks following Theorem~\ref{T:Mnaxio}, we now present a direct, though cumbersome, proof that in the axiomatization of $\mathsf{V}(\mathbf{M}_3')$ obtained by Theorem~\ref{T:Mnaxio}, the identity \eqref{T3} is superfluous, in accordance with Theorem~\ref{T:M3axio}.
Specifically, we show, without any reference to Theorem~\ref{T:M3axio}, that if $\mathbf{L}\in\mathcal{M}_{DM}$ satisfies \eqref{M3}, then $\mathbf{L}\in\mathcal{W}$ (Lemma~\ref{L:A1} below) and $\mathbf{L}$ satisfies the identity
\begin{equation}\label{E:XX}
(x\vee x'')''\approx x\vee x''
\end{equation}
(Lemma~\ref{L:A2}). In view of Lemma~\ref{L:neutral}, this ensures that every element of the form $\tau_3(a)=a\vee a''$ is neutral, and hence $\mathbf{L}$ satisfies the identity \eqref{T3}.

\begin{alemma}\label{L:A1}
If $\mathbf{L}\in\mathcal{M}_{DM}$ satisfies \eqref{M3}, then $\mathbf{L}\in\mathcal{W}_{DM}$.
\end{alemma}

\begin{proof}
Let $\mathbf{L}\in\mathcal{M}_{DM}$ and suppose that $a+b=0$ for some $a,b\in L$, i.e., $a'\wedge b=0$ and $a\wedge b'=0$. 
By \eqref{M3} we have
\begin{align*}
a\wedge (a'\vee b)\wedge (a\wedge b)'
&= a\wedge (a'\vee b)\wedge (a'\vee b')\wedge (b\vee b')\\ 
&\le (a\wedge a')\vee (a\wedge b)\vee (a\wedge b')\\
&= a\wedge b,
\end{align*}
whence
$a\wedge (a'\vee b)\wedge (a\wedge b)'\le a\wedge b\wedge (a\wedge b)'=0$, and so
\[
a\wedge (a'\vee b)\wedge (a\wedge b)'=0.
\]
Then, repeatedly using modularity, we obtain
\begin{align*}
b &= b\vee \big(a\wedge (a'\vee b)\wedge (a\wedge b)'\big)\\
&= \Big(b\vee \big(a\wedge (a\wedge b)'\big)\Big)\wedge (a'\vee b)\\
&= \Big(b\vee (a\wedge b)\vee \big(a\wedge (a\wedge b)'\big)\Big)\wedge (a'\vee b)\\
&= (b\vee a)\wedge (a'\vee b).
\end{align*}
Now,
$0=a'\wedge b=a'\wedge (a\vee b)\wedge (a'\vee b)=
a'\wedge (a\vee b)$, which implies 
$a=a\vee (a'\wedge (a\vee b))=a\vee b$ by modularity.
Thus $a\ge b$. Symmetrically, we obtain $b\le a$.
\end{proof}

\begin{alemma}\label{L:A2}
If $\mathbf{L}\in\mathcal{M}_{DM}$ satisfies \eqref{M3}, then it satisfies \eqref{E:XX}.
\end{alemma}

\begin{proof}
Suppose that $\mathbf{L}\in\mathcal{M}_{DM}$ satisfies \eqref{M3}. Let $a\in L$. Since $\mathbf{L}\in\mathcal{W}$ by Lemma~\ref{L:A1}, it will suffice to show that
$(a\vee a'')'''\wedge (a\vee a'')=0$, 
because then $(a\vee a'')''+(a\vee a'')=0$ implies $(a\vee a'')''=a\vee a''$.
Put 
\[
b:=(a\vee a'')'''\wedge (a\vee a'')
=a'''\wedge a'''''\wedge (a\vee a'').
\]

First, we note that
\begin{equation}\label{E:Aa}
b\wedge a'=0,
\end{equation}
as $b\le (a\vee a'')\wedge a'''$ yields
$b\wedge a'\le (a\vee a'')\wedge a'''\wedge a'=(a\vee a'')\wedge (a''\vee a)'=0$.
Besides, it is obvious that
\begin{equation}\label{E:Ab}
b\wedge a''=0.
\end{equation}
We then have
\begin{align*}
b\wedge a 
&= (b\wedge a)\vee (b\wedge a')\vee (b\wedge a'')
&& \text{by \eqref{E:Aa} and \eqref{E:Ab}}\\
&\ge b\wedge (a\vee a')\wedge (a\vee a'')\wedge (a'\vee a'')
&& \text{by \eqref{M3}}\\
&= b\wedge (a\vee a'')=b, 
\end{align*}
and so
\begin{equation}\label{E:Ac}
b\le a \quad\text{and}\quad b''\le a''.
\end{equation}

From $b\le (a\vee a'')'''=(a''\vee a'''')'$ and \eqref{E:Ac} we obtain 
$b\wedge (b''\vee a'''')
\le (a''\vee a'''')'\wedge (a''\vee a'''')=0$,
thus
\begin{equation}\label{E:Ad}
(a''''\vee b'')\wedge b=0.
\end{equation}

Now, using modularity together with \eqref{E:Ad} and \eqref{E:Ab}, we get
\begin{align*}
a''''\wedge (b\vee b'')
&=a''''\wedge (a''''\vee b'')\wedge (b\vee b'')\\
&=a''''\wedge \Big(b''\vee \big((a''''\vee b'')\wedge b\big)\Big)\\
&=a''''\wedge (b''\vee 0) && \text{by \eqref{E:Ad}}\\
&=(a''\wedge b)''=0, && \text{by \eqref{E:Ab}}
\end{align*}
i.e.,
\begin{equation}\label{E:Ae}
a''''\wedge (b\vee b'')=0.
\end{equation}
Similarly, using modularity together with \eqref{E:Ad} and \eqref{E:Ae}, we get
\begin{align*}
(a''''\vee b)\wedge b''
&= (a''''\vee b)\wedge (b''\vee 0)\\
&= (a''''\vee b)\wedge \Big(b''\vee \big((a''''\vee b'')\wedge b\big)\Big) && \text{by \eqref{E:Ad}}\\
&= (a''''\vee b)\wedge (b''\vee b)\wedge (a''''\vee b'')\\
&= \Big(a''''\vee \big((a''''\vee b'')\wedge b\big)\Big)\wedge (b''\vee b) \\
&= (a''''\vee 0)\wedge (b''\vee b) && \text{by \eqref{E:Ad}}\\
&= a''''\wedge (b''\vee b)=0, && \text{by \eqref{E:Ae}}
\end{align*}
thus
\begin{equation}\label{E:Af}
(a''''\vee b)\wedge b''=0.
\end{equation}
Finally, we have
\begin{align*}
0 &= (b''\wedge b')\vee (b\wedge a')''\vee \big(b''\wedge (a''''\vee b)\big) && \text{by \eqref{E:Aa} and \eqref{E:Af}}\\
&= (b''\wedge b')\vee (b''\wedge a''')\vee \big(b''\wedge (a''''\vee b)\big)\\
&\ge b''\wedge (b'\vee a''')\wedge (b'\vee a''''\vee b)\wedge (a'''\vee a''''\vee b) && \text{by \eqref{M3}}\\
&=b''\wedge (b\wedge a'')'\\
&=b''\wedge 1=b'', && \text{by \eqref{E:Ab}}
\end{align*}
thus $b''=0$, whence $b=0$, as desired.
\end{proof}

As a final remark, \eqref{E:XX} can be regarded as an instance of
\begin{equation}\label{E:Atau}
(\tau_n(x))''\approx\tau_n(x).
\end{equation}
Although this identity holds in $\mathbf{M}_n'$ for every $n\ge 3$, it cannot replace \eqref{tn} in the axiomatization of $\mathsf{V}(\mathbf{M}_n')$ when $n\ge 4$. 
The reason is that, for $n\ge 4$, \eqref{mn} does not ensure that condition $a''=a$ is equivalent to $a$ being neutral (cf.\ Lemma~\ref{L:neutral}). 
Thus, \eqref{E:Atau} together with \eqref{mn} does not guarantee that elements of the form $\tau_n(a)$ are neutral, which is precisely the role of \eqref{tn}.

\section*{Funding}

This work has been supported by the bilateral project ``Orthogonality and Symmetry'' 
by the Austrian Science Fund (FWF): 
No. 10.55776/PIN5424624 (H.~L\"anger), 
and by the Czech Science Foundation (GA\v CR): 
No. 25-20013L (V.~Cenker, I.~Chajda and J.~K\"uhr).

\end{document}